\documentclass[12pt, a4paper]{amsart}
\usepackage{amsmath}
\usepackage{geometry,amsthm,graphics,tabularx,amssymb,
shapepar}
\usepackage{amscd}
\usepackage[usenames]{color}

\usepackage[all,2cell,dvips]{xy}

\newcommand{\CN}{{\mathcal {N}}}

\newcommand{\Ad}{{\mathrm{Ad}}}
\newcommand{\Aut}{{\mathrm{Aut}}}

\newcommand{\GL}{{\mathrm{GL}}}

\newcommand{\Hom}{{\mathrm{Hom}}}

\newcommand{\Ind}{{\mathrm{Ind}}}

\newcommand{\Lie}{{\mathrm{Lie}}}

\newcommand{\rk}{{\mathrm{k}}}

\newcommand{\SL}{{\mathrm{SL}}}

\newcommand{\tr}{{\mathrm{tr}}}

\newcommand{\ad}{\operatorname{ad}}

\newcommand{\oB}{\operatorname{B}}

\newcommand{\oU}{\operatorname{U}}

\newcommand{\oZ}{\operatorname{Z}}

\newcommand{\oP}{\operatorname{P}}

\newcommand{\oN}{\operatorname{N}}

\newcommand{\g}{\mathfrak g}

\renewcommand{\k}{\mathfrak k}
\newcommand{\h}{\mathfrak h}
\newcommand{\p}{\mathfrak p}

\renewcommand{\a}{\mathfrak a}
\renewcommand{\b}{\mathfrak b}

\newcommand{\n}{\mathfrak n}
\renewcommand{\u}{\mathfrak u}

\renewcommand{\l}{\mathfrak l}

\newcommand{\s}{\mathfrak s}

\newcommand{\z}{\mathfrak z}

\renewcommand{\sl}{\mathfrak s \mathfrak l}
\newcommand{\gl}{\mathfrak g \mathfrak l}

\renewcommand{\rk}{\mathrm k}

\newcommand{\Z}{\mathbb{Z}}
\newcommand{\C}{\mathbb{C}}
\newcommand{\R}{\mathbb R}

\newcommand{\la}{\langle}
\newcommand{\ra}{\rangle}
\newcommand{\be}{\begin {equation}}
\newcommand{\ee}{\end {equation}}
\newcommand{\bp}{\begin {proof}}
\newcommand{\ep}{\end {proof}}
\newcommand{\bee}{\begin {equation*}}
\newcommand{\eee}{\end {equation*}}

\theoremstyle{Theorem}

\newtheorem{thm}{Theorem}[section]

\newtheorem{lemt}[thm]{Lemma}
\newtheorem{prpt}[thm]{Proposition}

\theoremstyle{Theorem}
\newtheorem{lem}{Lemma}[section]

\newtheorem{thml}[lem]{Theorem}

\newtheorem{prpl}[lem]{Proposition}

\theoremstyle{Theorem}
\newtheorem{prp}{Proposition}[section]

\newtheorem{lemp}[prp]{Lemma}
\newtheorem{thmp}[prp]{Theorem}

\theoremstyle{Plain}

\theoremstyle{Definition}
\newtheorem{dfn}{Definition}[section]

\newtheorem{dfnp}[prp]{Definition}
\newtheorem{dfnl}[lem]{Definition}

\newtheorem{prpd}[dfn]{Proposition}
\newtheorem{thmd}[dfn]{Theorem}
\newtheorem{lemd}[dfn]{Lemma}

\begin{document}

\title[Almost linear Nash groups]{Almost linear Nash groups}

\author [B. Sun] {Binyong Sun}
\address{Academy of Mathematics and Systems Science\\
Chinese Academy of Sciences\\
Beijing, 100190,  China} \email{sun@math.ac.cn}


\thanks{Supported by NSFC Grant 11222101.}

\begin{abstract}
A Nash group is said to be almost linear if it has a Nash
representation with finite kernel. Structures and basic properties
of these groups are studied.
\end{abstract}

 \maketitle
\tableofcontents

\section{Introduction}

Basic notions and properties concerning Nash manifolds are reviewed in Section \ref{secnash}. In this introduction, we introduce some basic notions concerning Nash groups. See Section \ref{secng} for more details.

A Nash group is a group which is simultaneously a Nash
manifold so that all group operations are Nash maps. A Nash homomorphism is a group homomorphism between two Nash groups which
is simultaneously a Nash map. If a  Nash homomorphism is bijective, then its inverse is also a Nash homomorphism. In this case, we say that the Nash homomorphism is a Nash isomorphism. Two Nash groups are said to be Nash isomorphic to each other if there exists a Nash isomorphism between them.

Give a subgroup of a Nash group $G$, if it is semialgebraic, then it is automatically a closed Nash submanifold of $G$  (Proposition \ref{subg}). In this case, we call it a Nash subgroup of $G$. A Nash subgroup is canonically a Nash group.

As usual, all finite dimensional real representations of Lie groups are assumed to be continuous.  A Nash representation is a finite
dimensional real representation of a Nash group so that the action
map is a Nash map.
\begin{dfn}
A Nash group is said to be almost linear if it has a Nash representation with
finite kernel.
\end{dfn}

Almost linear Nash groups form a nice class of mathematical objects:
their structures are simpler than that of general Lie groups; in the
study of infinite dimensional representation theory, they are more
flexible than linear algebraic groups. Although there is a vast
literature on Lie groups and linear algebraic groups, it seems that
almost linear Nash groups have not been systematically studied (see
\cite{Sh2} for a brief introduction to Nash groups). The goal of
this article is to provide a detailed study of structures of almost
linear Nash groups, for possible later reference. The structure
theory of almost linear Nash groups is parrel to that of linear
algebraic groups. However, we try to avoid the language of algebraic
geometry to keep the article as elementary as possible. In what
follows, we summarize some basic results about almost linear Nash
groups which are either well known or will be proved in this
article.

It is clear that a Nash subgroup of an almost linear Nash group is
an almost linear Nash group. The product of two almost linear Nash
groups is an almost linear Nash group. By the following proposition,
the quotient of an almost linear Nash group by a Nash subgroup is
canonically an affine Nash manifold, and the quotient by a normal
Nash subgroup is canonically an almost linear Nash group.

\begin{prpd}\label{quotient0}
Let $G$ be an almost linear Nash group, and let $H$ be a Nash
subgroup of it. Then there exists a unique Nash structure on the
quotient topological space $G/H$ which makes the quotient map
$G\rightarrow G/H$ a submersive Nash map. With this Nash structure,
$G/H$ becomes an affine Nash manifold, and the left translation map
$G\times G/H\rightarrow G/H$ is a Nash map. Furthermore, if $H$ is a
normal Nash subgroup of $G$, then the topological group $G/H$
becomes an almost linear Nash group under this Nash structure.
\end{prpd}

For each normal Nash subgroup $H$ of an almost linear Nash group $G$, the Nash group $G/H$ is called a Nash quotient group of $G$.

There are three classes of almost linear Nash groups which are basic
to the general structure theory, namely, elliptic Nash
groups, hyperbolic Nash groups and unipotent Nash groups.

\begin{dfn}\label{ehu0}
A Nash group is said to be elliptic if it is almost linear and
compact. It is said to be hyperbolic if it is Nash isomorphic to
$(\R_+^\times)^n$ for some $n\geq 0$. It is said to be unipotent if
it has a faithful Nash representation so that all group elements act
as unipotent linear operators.
\end{dfn}

Here and as usual, $\R_+^\times$ denotes the set of positive real numbers.
It is a Nash group in the obvious way. Recall that a linear operator $x$ on a finite dimensional vector space is said to be unipotent if $x-1$ is nilpotent.

There is no need to say that a Nash group is almost linear if it is elliptic, hyperbolic or unipotent.

\begin{dfn}\label{defehue}
An element of an almost linear Nash group $G$ is said to be
elliptic, hyperbolic, or unipotent if it is contained in a Nash
subgroup of $G$ which is elliptic, hyperbolic, or unipotent,
respectively.
\end{dfn}

Definitions \ref{ehu0} and \ref{defehue} are related as follows.
\begin{prpd}\label{criehu}
An almost linear Nash group is elliptic, hyperbolic, or unipotent if
and only if all of its elements are elliptic, hyperbolic, or
unipotent, respectively.
\end{prpd}

In general, we have the following
\begin{prpd}\label{thmehu}
Let $G$ be an almost linear Nash group. If $G$ is elliptic, hyperbolic or unipotent, then all Nash subgroups and all Nash quotient groups of $G$ are elliptic, hyperbolic or unipotent, respectively.
If $G$ has a normal Nash subgroup $H$ so that  $H$ and $G/H$ are  both elliptic, both hyperbolic or both unipotent, then $G$ is  elliptic, hyperbolic or unipotent, respectively.
\end{prpd}

Concerning elliptic Nash groups, we have
\begin{thmd}\label{thmc}
The followings hold true.
\begin{itemize}
  \item Every compact Lie group has a unique Nash structure on its underlying topological space which makes it an almost linear Nash group.
  \item Every continuous homomorphism from an elliptic Nash group to an almost linear Nash group is a Nash homomorphism.
  \item Every compact subgroup of an almost linear Nash group is a Nash subgroup.
  \end{itemize}
\end{thmd}

Theorem \ref{thmc} implies that the category of elliptic Nash groups is isomorphic to the category of compact Lie groups.

Recall that a subgroup of a Lie group $G$ is said to be analytic if
it is equal to the image of an injective Lie group homomorphism from
a connected Lie group to $G$. Every analytic subgroup is canonically
a Lie group. (This is implied by \cite[Theorem 1.62]{Wa}.)

For unipotent Nash groups, we have
\begin{thmd}\label{thmu}
The followings hold true.
\begin{itemize}
  \item As a Lie group, every unipotent Nash group is connected, simply connected and nilpotent.
    \item Every connected, simply connected, nilpotent Lie group has a unique Nash structure on its underlying topological space which makes it a unipotent  Nash group.
  \item Every continuous homomorphism between two unipotent Nash groups is a Nash homomorphism.
  \item Every analytic subgroup of a unipotent Nash group is a Nash subgroup.
  \end{itemize}
  \end{thmd}

Theorem \ref{thmu} implies that the category of unipotent Nash groups is isomorphic to the category of connected, simply connected, nilpotent Lie groups.
Recall that the later category is equivalent to the category of finite dimensional nilpotent real Lie algebras.

For every $r\in \mathbb Q$, the map
\[
  \R_+^\times \rightarrow \R_+^\times,\quad x\mapsto x^r
\]
is a Nash homomorphism from $\R_+^\times$ to itself. Conversely, all  Nash homomorphisms from $\R_+^\times$ to itself are of this form.
We view the abelian group $\R_+^\times$ as a right $\mathbb Q$-vector space so that
\[
  \textrm{the scalar multiplication $x\cdot r:=x^r$,}
\]
for all $x\in \R_+^\times$ and $r\in \mathbb Q$. Note that for every finite dimensional left $\mathbb Q$-vector space $E$, $\R^\times_+\otimes_\mathbb Q E$ is obviously a hyperbolic Nash group.
Moreover, we have
\begin{thmd}\label{thmh}
The functor
\[
  A\mapsto \Hom(\R^\times_+, A)
\]
establishes an equivalence from the category of hyperbolic Nash groups to the category of finite dimensional left $\mathbb Q$-vector spaces. It
has a quasi-inverse
\[
  E\mapsto \R^\times_+\otimes_\mathbb Q E.
\]
\end{thmd}

Here and henceforth, for any two Nash groups $G_1$ and $G_2$,
$\Hom(G_1,G_2)$ denotes the set of all Nash homomorphisms from $G_1$
to $G_2$. It is obviously an abelian group when $G_2$ is abelian.
The abelian group $\Hom(\R^\times_+, A)$ of Theorem \ref{thmh} is a
left $\mathbb Q$-vector space since $\R_+^\times$ is a right
$\mathbb Q$-vector space:
\[
  (r\cdot \phi)(x):=\phi(x^r),\quad r\in \mathbb Q,\, \phi\in \Hom(\R^\times_+, A),\, x\in \R_+^\times.
\]

Parallel to Jordan decompositions for linear algebraic groups, we have
\begin{thmd}\label{thm2}
  Every element $x$ of an almost linear Nash group $G$ is uniquely of the form $x=e h u$ such that $e\in G$ is elliptic, $h\in G$ is hyperbolic, $u\in G$ is unipotent, and they pairwise commute with each other.
      \end{thmd}

We call the equality  $x=ehu$ of Theorem \ref{thm2} the Jordan decomposition of $x$. In Section \ref{secjordan}, Jordan decompositions at Lie algebra level are also discussed.

Besides elliptic Nash groups, hyperbolic Nash groups and unipotent Nash groups, there are two other classes of Nash groups which are important to general structure theory, namely, reductive Nash groups and exponential Nash groups.
\begin{dfn}
A Nash group is said to be reductive if it has a completely
reducible Nash representation with finite kernel. It is said to be exponential if it is almost linear and has no non-trivial elliptic element.
\end{dfn}
Here and as usual, a representation is said to be completely
reducible if it is a direct sum of irreducible subrepresentations,
or equivalently, if each subrepresentation of it has a complementary
subrepresentation.

A general reductive Nash group is  more or less the direct product
of two reductive Nash groups of particular type, namely, a
semisimple Nash group and a Nash torus.
\begin{dfn}\label{d12}
A Nash group or a Lie group is said to be semisimple if its Lie algebra is semisimple. A Nash torus is a Nash group which is Nash isomorphic to $\mathbb S^m\times (\R_+^\times)^n$ for some $m,n\geq 0$.
\end{dfn}
Here $\mathbb S$ denotes the Nash group of complex numbers of modulus  one.

Concerning semisimple Nash groups, we have
\begin{thmd}\label{thmcs}
The followings hold true.
\begin{itemize}
\item Every semisimple Nash group is almost linear.
  \item Every semisimple Nash group has finitely many
connected components, and its identity connected component has a finite center.
\item Let $G$ be a semisimple Lie group which has finitely many connected
components, and whose  identity connected component has a finite
center. Then there exists a unique Nash structure on the underlying
topological space of $G$ which makes $G$ a Nash group.
  \item  Every continuous homomorphism from a semisimple Nash group to an almost linear Nash group is a Nash homomorphism.
  \item Every semisimple  analytic subgroup of an almost linear Nash group is a Nash subgroup.
  \end{itemize}
\end{thmd}

Theorem \ref{thmcs} implies that the category of semisimple Nash
groups is isomorphic to the category of semisimple Lie groups which
have finitely many connected components, and whose identity
connected component has a finite center.

For every almost linear Nash group $G$, define its unipotent radical to be
\[
 \mathfrak U_G:=\textrm{the identity connected component of } \bigcap_{\pi} \ker \pi,
\]
where $\pi$ runs through all irreducible Nash representations of
$G$. This is the largest normal unipotent Nash subgroup of $G$
(Proposition \ref{0unir}).

We have the following theorem concerning reductive Nash groups.

\begin{thmd}\label{thmred}
The followings are
equivalent for an almost linear Nash group $G$.
\begin{itemize}
  \item[(a)] It is reductive.
  \item[(b)] All Nash representations of $G$ are completely reducible.
  \item[(c)] The unipotent radical of $G$ is trivial.
  \item[(d)] For some Nash representations of $G$ with finite kernel, the attached trace form on the Lie algebra of $G$ is non-degenerate.
  \item[(e)] For every Nash representation of $G$ with finite kernel, the attached trace form on the Lie algebra of $G$ is non-degenerate.
  \item[(f)] The identity connected component of $G$ is reductive.
   \item[(g)] There exist a connected semisimple Nash group $H$, a Nash torus $T$, and a Nash homomorphism $H\times T\rightarrow G$ with finite kernel and open image.
\end{itemize}

\end{thmd}

Here for every Nash representation $V$ of a Nash group $G$, the attached trace form $\la\,,\,\ra_\phi$ on the Lie algebra $\g$ of $G$ is defined by
\[
   \la x,y\ra_\phi:=\tr(\phi(x)\phi(y)),\quad x,y\in \g,
\]
where $\phi: \g\rightarrow \gl(V)$ denotes the differential of the
representation $V$ of $G$. Here and as usual, $\gl(V)$ denotes the algebra of
all linear endomorphisms of $V$; and as quite often, when no
confusion is possible, we do not distinguish a representation with
its underlying vector space.

Denote by $\oB_n(\R)$ the Nash subgroup of $\GL_n(\R)$ consisting
all upper-triangular matrices with positive diagonal entries ($n\geq
0$). It is obviously an exponential Nash group.

\begin{thmd}\label{thmexp}
The followings are
equivalent for an almost linear Nash group $G$.
\begin{itemize}
  \item[(a)] It is exponential;
  \item[(b)] It has no non-trivial compact subgroup.
  \item[(c)] It has no proper co-compact Nash subgroup.
  \item[(d)] The quotient $G/\mathfrak U_G$ is a hyperbolic Nash group.
  \item[(e)] It is Nash isomorphic to a Nash subgroup of $\oB_n(\R)$ for some $n\geq 0$.
   \item[(f)] The exponential map from the Lie algebra of $G$ to $G$ is a diffeomorphism.
   \item[(g)] Every Nash action of $G$ on every non-empty compact Nash manifold has a fixed point.
\end{itemize}

\end{thmd}

Here a Nash action means an action of a Nash group on a Nash manifold so that the action map is Nash.

The following theorem makes the structure theory of almost linear Nash groups extremely pleasant.

\begin{thmd}\label{conall}
Let $G$ be an almost linear Nash group.  Then every elliptic
(hyperbolic, unipotent, reductive or exponential) Nash subgroup of
$G$ is contained in a maximal one, and all maximal elliptic
(hyperbolic, unipotent, reductive or exponential) Nash subgroups of
$G$ are conjugate to each other in $G$.
\end{thmd}

A maximal reductive Nash subgroup of an almost linear Nash group $G$
is  called a Levi component of $G$.

\begin{thmd}\label{levd0}
Let $L$ be a Levi component of an almost linear Nash
group $G$.  Then $G=L\ltimes \mathfrak U_G$.
      \end{thmd}

The equality $G=L\ltimes \mathfrak U_G$ of Theorem \ref{levd0} is called a Levi decomposition of $G$.

\begin{thmd}\label{iwasawa}
Let $G$ be an almost linear Nash group. Let $K$ be a maximal elliptic Nash subgroup of $G$, and let $B$ be a maximal exponential Nash subgroup of $G$. Then the multiplication map $K\times B\rightarrow G$ is a Nash diffeomorphism.
\end{thmd}

Let $G, K$ and $B$ be as in Theorem \ref{iwasawa}. Let $A$ be a Levi component of $B$, which is a hyperbolic Nash group. Denote by $N$ the unipotent radical of $B$. Then by Theorems \ref{levd0} and \ref{iwasawa}, we have that $G=KAN$. This is called an Iwasawa
decomposition of $G$.

The author thanks Masahiro Shiota for helpful email correspondences, and for confirming Proposition \ref{coverm}.

\section{Nash manifolds}\label{secnash}

We begin with a review of basic concepts and properties of Nash
manifolds which are necessary for this article. See \cite{BCR, Sh}
for more details. Recall that a subset of  $\R^n$ ($n\geq 0$) is
said to be semialgebraic if it is a finite union of the sets of the
form
\[
  \{x\in \R^n\mid f_1(x)>0,\,f_2(x)>0,\cdots, f_r(x)>0, \, g_1(x)=g_2(x)=\cdots =g_s(x)=0\},
\]
where $r,s\geq 0$, $f_1, f_2,\cdots, f_r$ and $g_1,g_2,\cdots, g_s$
are real polynomial functions on $\R^n$. For $n=-\infty$, we define $\R^n$ to be the empty set, and its only subset is defined to be semialgebraic.
It is clear that the
collection of semialgebraic sets in $\R^n$ ($n\geq 0$ or $n=-\infty$) is closed under taking
finite union, finite intersection, and complement.

A map
$\varphi:X\rightarrow X'$ from a semialgebraic set $X\subset \R^n$
to a  semialgebraic set $X'\subset \R^m$ ($m\geq 0$ or $m=-\infty$) is said to be
semialgebraic if its graph is semialgebraic in $\R^{n+m}$.
Tarski-Seidenberg Theorem asserts that the image of a semialgebraic
set under a semialgebaic map is semialgebraic: if
$\varphi:X\rightarrow X'$ is semialgebraic, then $\varphi(X_0)$ is
semialgebraic for each semialgebaic set $X_0\subset X$. As an easy
consequence of Tarski-Seidenberg Theorem, we know that the
composition of two semialgebraic maps is also semialgebraic; and the inverse image of a semialgebraic set under a semialgebraic map is semialgebraic.

\begin{dfn}\label{dfnnash}
 A Nash structure on a topological space $M$ is an element $n\in\{-\infty, 0,1,2,\cdots\}$ together with a set $\mathcal{N}$ with the following properties:
\begin{enumerate}
\item[(a)]The set $\mathcal N$ is contained in $\oN(\R^n,M)$, where $\oN(\R^n,M)$ denotes the set of all triples $(\phi,U,U')$ such that $U$
 is an open semialgebraic subset of $\R^n$, $U'$ is an open subset of $M$, and $\phi:U\rightarrow U'$ is a homeomorphism.
\item[(b)]
 Every two elements $(\phi_1,U_1,U'_1)$ and $(\phi_2,U_2,U'_2)$ of $\mathcal N$ are
 Nash compatible, namely, the homeomorphism
 \[
  \phi_2^{-1}\circ \phi_1: \phi_1^{-1}(U_1'\cap U_2')\rightarrow
 \phi_2^{-1}(U_1'\cap U_2')
 \]
has semialgebraic domain and codomain, and is semialgebraic and
smooth.
\item[(c)]
  There are finitely many elements $(\phi_i,U_i,U'_i)$ of $\mathcal N$,
 $i=1,2,\cdots, r$ ($r\geq 0$), such that
\[
   M=U'_1\cup U'_2\cup \cdots \cup U_r'.
\]
  \item[(d)]
  For every element of $\oN(\R^n,M)$, if it is Nash compatible with all elements of $\mathcal N$, then itself is an element of $\mathcal N$.
\item[(e)] If $M$ is empty, then $n=-\infty$.
  \end{enumerate}
\end{dfn}

The following lemma is routine to  check.
\begin{lemd}\label{defnashs0}
With the notation as in Definition \ref{dfnnash}, let
\[
 \mathcal N_0=\{(\phi_i, U_i, U_i')\mid i=1,2,\cdots,r\}
 \]
be a finite subset of $\oN(\R^n,M)$ whose elements are pairwise Nash compatible with each other. If $M$ is non-empty and
\[
   M=U'_1\cup U'_2\cup \cdots \cup U_r',
\]
then together with $n$, the set
\[
  \{(\phi,U,U')\in \oN(\R^n,M)\mid (\phi,U,U')\textrm{ is  Nash compatible with all elements of $\mathcal N_0$}\}
\]
is a Nash structure on $M$.
\end{lemd}

A Nash manifold is defined to be a Hausdorff topological space together with a Nash structure on it.
The element $n$ in Definition \ref{dfnnash} of the Nash structure  is called the dimension of the Nash manifold;
an element of $\mathcal N$ in Definition \ref{dfnnash} of the Nash structure is called  a Nash chart of the Nash manifold.

\begin{dfn}
A continuous map $\varphi: M\rightarrow N$ between Nash manifolds is
called a Nash map if for all Nash charts $(\phi,U,U')$ of $M$ and
$(\psi,V,V')$ of $N$, the set $\phi^{-1}(U'\cap \varphi^{-1}(V'))$ is semialgebraic, and the map
\[
   \psi^{-1}\circ\varphi\circ \phi: \phi^{-1}(U'\cap \varphi^{-1}(V')) \rightarrow V
\]
is semialgebraic and smooth.
\end{dfn}

It is clear that every Nash manifold is a smooth manifold, and every
Nash map is a smooth map. The composition of two Nash maps is
certainly a Nash map.

\begin{dfn}
A subset $X$ of a Nash manifold $M$ is said to be semialgebraic if
$\phi^{-1}(X\cap U')$ is semialgebraic  for every Nash
 chart $(\phi,U,U')$ of $M$.

 \end{dfn}

As in the case of $\R^n$, the
collection of semialgebraic sets in a Nash manifold is closed under taking
finite union, finite intersection, and complement. Tarski-Seidenberg Theorem easily implies the following

\begin{lemd}\label{images}
Let $\varphi: M\rightarrow N$ be a Nash map of Nash manifolds. Then for each semialgebraic subset $X$ of $M$, the image $\varphi(X)$ is a
semialgebraic subset of $N$; and for each semialgebraic subset $Y$ of $N$, the inverse image $\varphi^{-1}(Y)$ is a
semialgebraic subset of $M$.
\end{lemd}

The following is a useful criteria for a continuous map to be Nash.

\begin{prpd}\label{crinash}
Let $\varphi: M\rightarrow N$ be a continuous map of Nash manifolds. Then $\varphi$ is a Nash map if and only if
\begin{itemize}
  \item for each semalgebraic open subset $Y$ of $N$, the inverse image $\varphi^{-1}(Y)$ is semialgebraic in $M$; and
  \item for every $x\in M$, there are Nash charts $(\phi,U,U')$ of $M$
and  $(\psi,V,V')$ of $N$ such that  $x\in U'$, $\varphi(U')\subset
V'$, and the map
\[
   \psi^{-1}\circ\varphi\circ \phi: U\rightarrow V
\]
is semialgebraic and smooth.
\end{itemize}
\end{prpd}
\begin{proof}
This is an easy consequence of \cite[Proposition 8.1.8]{BCR}.
\end{proof}

The following lemma will be used for several times.
\begin{lemd}\label{subm0}
Let $\varphi: M\rightarrow M'$ be a surjective submersive Nash map
of Nash manifolds. Let $N$ be a Nash manifold and let $\psi:
M'\rightarrow N$ be a map. Then $\psi$ is a Nash map if and only if
$\psi\circ \varphi$ is a Nash map.
\end{lemd}
\begin{proof}
The ``only if" part of the Lemma is obvious. Using Proposition
\ref{crinash} and Lemma \ref{images}, the ``if" part holds because
the map $\varphi$ has local Nash sections.
\end{proof}

Given a Nash map, if it is a diffeomorphism as a map of smooth
manifolds, then its inverse is also a Nash map. In this case, we
call the Nash map a Nash diffeomorphism. Two Nash manifolds are said
to be Nash diffeomorphic to each other if there exists a Nash
diffeomorphism between them.

\begin{dfn}A semialgebraic locally closed submanifold of a Nash manifold $M$
 is called a Nash submanifold of $M$.

 \end{dfn}

 In this article, all locally closed submanifolds of a smooth manifold are assumed to be equidimensional.
Every Nash submanifold is automatically a Nash manifold:
\begin{prpd}\label{sub}
Let $X$ be a Nash submanifold of a Nash manifold $M$. Then there
exists a unique Nash structure on the topological space $X$ which
makes the inclusion $X\hookrightarrow M$ an immersive Nash map.
\end{prpd}

We say that a Nash map $\varphi: M\rightarrow N$ is a Nash embedding if $\varphi(M)$ is a Nash submanifold of $N$, and the induced map $\varphi: M\rightarrow \varphi(M)$ is a Nash diffeomorphism.

The product of two Nash manifolds is again a Nash manifold:
\begin{prpd}\label{prod}
Let $M$ and $N$ be two Nash manifolds. Then there exists a unique Nash
structure on the topological space $M\times N$ which makes the
projections $M\times N\rightarrow M$ and $M\times N\rightarrow N$
submersive Nash maps.
\end{prpd}

Both Propositions \ref{sub} and \ref{prod} are standard. We shall
not go to their  proofs.

The following lemma is obvious.
\begin{lemd}\label{cnash}
Let $\varphi: M\rightarrow N$ be a smooth map of Nash manifolds. Then
$\varphi$ is a Nash map if and only if its graph is semialgebraic in
$M\times N$.
\end{lemd}

Lemmas  \ref{cnash}, \ref{images}, and the following basic result will be used without further explicit mention.
\begin{lemd}\label{dimen20} (cf. \cite[Theorem 2.23]{Co})
Every semialgebraic subset of a Nash manifold has only finitely many connected components and each of them is semialgebraic.
\end{lemd}

Recall the following

\begin{lemd}\label{coverrn0} (cf. \cite[Remark I.5.12]{Sh})
Every Nash manifold of dimension $n$ ($n\geq 0$) is covered by
finitely many open Nash submanifolds which are Nash diffeomorphic to $\R^n$.
\end{lemd}

Using Lemma \ref{dimen20} and Lemma \ref{coverrn0}, it is easy to prove the following
\begin{lemd}\label{subm00}
Let $\varphi: M\rightarrow M'$ be a submersive Nash map
of Nash manifolds. Assume that $\varphi$ is a finite
fold covering map as a map of topological spaces. Let $N$ be a Nash manifold and let $\psi:
N\rightarrow M$ be a continuous map. Then $\psi$ is a Nash map if and only if
$\varphi \circ \psi$ is a Nash map.
\end{lemd}

By the following proposition, a finite fold cover of
a Nash manifold is a Nash manifold:

\begin{prpd}\label{coverm}
Let $N$ be a Nash manifold. Let $M$ be a topological space and let $\varphi:M\rightarrow N$ be a finite
fold covering map of topological spaces. Then there exists a unique Nash
structure on $M$ which makes $\varphi$ a submersive Nash map.
\end{prpd}
\begin{proof}
This is known to experts. We sketch a proof for the lack of
reference. First note that $M$ is Hausdorff since $N$ is Hausdorff. Since the proposition is trivial when $M$ is an empty set,
we assume that $M$ is non-empty. The uniqueness assertion of the proposition is a direct consequence of Lemma \ref{subm00}. In what follows, we construct a Nash structure on $M$ which makes $\varphi$ a submersive Nash map.

Write ($n$, $\CN_N$) for the Nash structure on $N$.
Put
\[
  \CN_M':=\bigcup_{(\phi,U,U')\in \CN_N, \textrm{ $\,$ with $U$  connected and simply connected}}\CN_\phi,
\]
where
\[
  \CN_\phi:=\{(\psi,U,U'')\in \oN(\R^n, M)\mid \psi \textrm{ lifts the homeomorphism }\phi:U\rightarrow U'\}.
\]
One checks that all elements in $\CN_M'$ are pairwise Nash
compatible. Lemma \ref{coverrn0} implies that the set $\CN_M'$ has
property (c) of Definition \ref{dfnnash}. Denote by $\mathcal N_M$
the set of all elements in $\oN(\R^n, M)$ which are Nash compatible
with all elements of $\CN_M'$. Lemma \ref{defnashs0} implies that
$(n,\mathcal N_M)$ is a Nash structure on $M$. With this Nash
structure, $\varphi$ is clearly a submersive Nash map.
\end{proof}

Every finite dimensional real vector space is obviously a Nash manifold. A Nash manifold is said to be affine if it is Nash diffeomorphic to a Nash submanifold of some finite dimensional
 real vector space. It is known that
every affine Nash manifold is actually Nash diffeomorphic to a
closed Nash submanifold of some finite dimensional real vector
space. (cf. \cite[2.22]{Sh2}). It is clear that a Nash submanifold
of an affine Nash manifold is an affine Nash manifold; the products
of two affine Nash manifolds is an affine Nash manifold. The
following criterion implies that a finite fold cover of an affine
Nash manifold is an affine Nash manifold:
\begin{prpd} \label{Sh} (\cite[Proposition III.1.7]{Sh})
Let $M$ be a Nash manifold of dimension $n\geq 0$. Then $M$ is affine if and only if for every $x\in M$,
there is a Nash map $M\rightarrow \R^n$ which is submersive at $x$.
\end{prpd}

Projective spaces form an important family of affine Nash manifolds: For each finite dimensional real vector space $V$,
 the set $\oP(V)$ of all one dimensional subspaces of $V$ is naturally an affine Nash manifold (cf. \cite[Theorem 3.4.4]{BCR}).

For each semialgebraic subset $X$ of a Nash manifold $M$, define its dimension
\[
  \begin{array}{rcl}
  \dim X&:=&\max \{d\in \{-\infty, 0,1,2,\cdots\}\mid \\
  &&\textrm{$\quad  \quad X$ contains a Nash submanifold of $M$ of dimension $d$} \}.
  \end{array}
\]
The following properties of the dimensions of semialgebraic sets are obvious (\cite[Page 56]{Co}):
the dimension of the union of finitely many semialgebraic sets is the maximum
of the dimensions of these semialgebraic sets; the dimension of a finite product of semialgebraic sets is the sum of their dimensions. The following basic facts concerning dimensions of
semialgebraic sets are well known.

\begin{prpd}\label{dimen2} (cf. \cite[Proposition 3.16 and Theorem 3.20]{Co}) The followings hold true.
\begin{itemize}
  \item The closure $\bar X$ of a semialgebraic set $X$ in a Nash manifold is semialgebraic. Moreover, $\dim \bar X=\dim X$; and $\dim \bar X\setminus X<\dim X$ whenever $X$ is non-empty.
   \item Each semialgebraic subset of a finite dimensional real vector space has the same dimension as its Zariski closure.
 \end{itemize}
\end{prpd}
Note that all Zariski closed subsets of a finite dimensional real vector space are semialgebraic.

For a semialgebraic set $X$ of a Nash manifold $M$, an element $x\in X$ is said to be smooth of dimension $d\geq 0$ if there is a semialgebraic open neiborhood $U$ of $x$ in $M$ such that $X\cap U$ is a $d$-dimensional Nash submanifold of $M$. Note that $X$ is a Nash submanifold of $M$ if and only if all points of it are smooth of dimension $\dim X$.

\begin{lemd}\label{smooth}
(cf. \cite[Proposition 5.53]{BOR}) Let $X$ be a non-empty semialgebraic subset of a Nash manifold $M$. Then $X$ has a point which is smooth of dimension $\dim X$.
\end{lemd}

\section{Nash groups and almost linear Nash groups}\label{secng}

In this section, we introduce some generalities on Nash groups and almost linear Nash groups.
\begin{dfn}
A Nash group is a Hausdorff topological group $G$, equipped with a Nash structure on its underlying topological space so that both the multiplication map $G\times G\rightarrow G$ and the inversion map $G\rightarrow G$ are Nash maps between Nash manifolds. A Nash homomorphism between two Nash groups is a group
homomorphism between them which is simultaneously a Nash map.
\end{dfn}

The following basic result will be used freely without further explicit mention.
\begin{prpd}\label{subg}
Every semialgebraic subgroup of a Nash group $G$ is a closed Nash
submanifold of $G$.
\end{prpd}
\begin{proof}
This is well know. We sketch a proof for convenience of the reader.
Let $H$ be a semialgebraic subgroup of $G$. Lemma \ref{smooth}
implies that $H$ is a Nash submanifold of $G$. In particular, $H$ is
locally closed, and is thus an open subgroup of its closure $\bar
H$. Therefore $H$ is also closed in $\bar H$. This proves the
Proposition. (Recall that every closed subgroup of a Lie group is a
submanifold.)
\end{proof}

In view of Proposition \ref{subg}, a semialgebraic subgroup of a Nash group $G$ is also called a Nash subgroup of $G$.

By Lemma \ref{images}, we have
\begin{prpd}\label{subgi}
The image of a Nash homomorphism $\varphi:G\rightarrow G'$ is a Nash subgroup of
$G'$. In particular, it is closed in $G'$.
\end{prpd}

 It is clear that a Nash subgroup of a Nash group is a Nash group, and the product of two Nash groups is a Nash group.
 Proposition \ref{coverm} implies that a finite fold covering group of a Nash group is a Nash group:
\begin{prpd}\label{subg2}
Let $G$ be a topological group and let $G'$ be a Nash group. Let
$G\rightarrow G'$ be a group homomorphism which is simultaneously a
finite fold topological covering map. Equip on $G$ the Nash
structure which makes $G\rightarrow G'$ a submersive Nash map. Then
$G$ becomes a Nash group.
\end{prpd}
\begin{proof}
Using Lemma \ref{subm00}, this is routine to check.
\end{proof}

Note that there is no strictly decreasing infinite sequence of Nash
subgroups of a Nash group. Consequently, we have
\begin{prpd}\label{subgi}
Let $G$ be a Nash group and let $\{G_i\}_{i\in I}$ be a family of
Nash subgroups of $G$. Then
\[
  \bigcap_{i\in I} G_i=\bigcap_{i\in I_0} G_i
\]
for some finite subset $I_0$ of $I$. Consequently, the intersection
of an arbitrary family of Nash subgroups of $G$ is again a Nash
subgroup of $G$.
\end{prpd}

By a Nash action of a Nash group $G$ on a Nash manifold $M$, we mean a group action $G\times M\rightarrow M$ which is simultaneously a Nash map.
Using Lemma \ref{smooth}, we know that each $G$-orbit of a Nash action $G\times M\rightarrow M$ is a Nash submanifold of $M$.

The analog of the following proposition for algebraic groups is proved in \cite[Chapter I, Proposition 1.8]{Bor91}.

\begin{prpd}\label{closed}
Let $G$ be Nash group with a Nash action on a non-empty Nash
manifold $M$. Then each $G$-orbit in $M$ of minimal dimension is
closed.
\end{prpd}
\begin{proof}
For each non-closed $G$-orbit $O$ in $M$, there is an orbit $O'$ in $\bar O\setminus O$, where $\bar O$ denotes the closure of $O$ in $M$. Then $\dim O'<\dim O$ by the first assertion of Proposition \ref{dimen2}. Therefore $O$ is not of minimal dimension.
\end{proof}

A finite dimensional real representation $V$ of a Nash group $G$ is
called a Nash representation if the action map $G\times
V\rightarrow V$ is a Nash map. This is equivalent to saying that the
corresponding homomorphism $G\rightarrow \GL(V)$ is a Nash homomorphism. Recall from the introduction that a Nash
group is said to be almost linear if it has a Nash representation
with finite kernel.

In this article, we use a superscript ``$\,\,^\circ\,$" to indicate
the identity connected component of a Nash group.

\begin{prpd}\label{induction}
A Nash group $G$ is almost linear if and only if $G^\circ$ is so.
\end{prpd}
\begin{proof}
The ``only if" part is trivial. Assume that $G^\circ$ is almost
linear. Let $V_0$ be a Nash representation of $G^\circ$ with finite
kernel. Put
\[
  V:=\Ind_{G^\circ}^G V_0:=\{f:G\rightarrow V_0\mid f(g_0 g)=g_0.f(g),\,g_0\in G^\circ,\,g\in G\}.
\]
Under right translations, this is a Nash representation of $G$ with
finite kernel.
\end{proof}

To treat quotient spaces of almost linear Nash groups, recall the following
\begin{prpd}\label{chev}
Let $G$ be an almost linear Nash group, and let $H$ be a Nash
subgroup of it.
\begin{itemize}
\item There exist a Nash representation $V$ of $G$,
and a one dimensional subspace $V_1\subset V$ such that the
stabilizer
\[
   \{g\in G\mid g.V_1=V_1\}
\]
contains $H$ as an open subgroup.
 \item If $H$ is normal, then there exists a Nash representation of $G$ whose kernel contains $H$ as an open subgroup.
\end{itemize}
\end{prpd}
\begin{proof}
Using the second assertion of Proposition \ref{dimen2}, this is an easy consequence of Chevalley's Theorem (cf. \cite[Theorem 11.1.13]{GW}).
\end{proof}

Also recall the following well know

\begin{lemd}\label{warner}(cf. \cite[Theorem 3.62]{Wa}
 Let $G\times M\rightarrow M$ be a transitive smooth  action of a Lie group $G$ on a smooth manifold $M$. Then for each $x\in M$, the map
      \[
        G/G_x\rightarrow M,\quad g\mapsto g.x
      \]
      is a diffeomorphism. Here  $G_x:=\{g\in G\mid g.x=x\}$, and the quotient topological space $G/G_x$ is equipped with the manifold structure so that the quotient map $G\rightarrow G/G_x$ is smooth and submersive. Consequently, all surjective Lie group homomorphisms are submersive.
\end{lemd}

Here and as usual, all Lie groups and smooth manifolds are assumed to be Hausdorf and second countable as topological spaces.

\begin{prpd}\label{quotient}
Let $G$ be an almost linear Nash group, and let $H$ be a Nash
subgroup of it. Then there exists a unique Nash structure on the
quotient topological space $G/H$ which makes the quotient map
$G\rightarrow G/H$ a submersive Nash map. With this Nash structure,
$G/H$ becomes an affine Nash manifold, and the left translation map
$G\times G/H\rightarrow G/H$ is a Nash map. Furthermore, if $H$ is a
normal Nash subgroup of $G$, then the topological group $G/H$
becomes an almost linear Nash group under this Nash structure.
\end{prpd}
\begin{proof}
Uniqueness of such Nash structures is implied by Lemma
\ref{subm0}. Let $V$ and $V_1$ be as in the first assertion of Proposition \ref{chev}. The projective space $\oP(V)$, which is naturally a Nash manifold, carries the induced Nash action of $G$.  The image of the map
\[
  \varphi: G/H\rightarrow \oP(V), \quad gH\mapsto g. V_1
\]
is a $G$-orbit in $\oP(V)$, and is thus a Nash submanifold of $\oP(V)$. It is affine since $\oP(V)$ is an affine Nash manifold. Lemma \ref{warner} implies that the map
\begin{equation}\label{ghop}
  \varphi: G/H\rightarrow \varphi(G/H), \quad gH\mapsto g. V_1
\end{equation}
is a finite fold topological covering map.
Using Proposition \ref{coverm}, we equip  on $G/H$ the Nash structure which makes the map \eqref{ghop} a submersive Nash map. Then by Proposition \ref{Sh}, $G/H$ is an affine Nash manifold,
and Lemma \ref{subm00} implies that the left translation map $G\times G/H\rightarrow G/H$ is a Nash map.

Now assume that $H$ is normal. Using the second assertion of Proposition \ref{chev}, we get a Nash homomorphism
\be \label{vggln}
  \psi: G\rightarrow \GL_n(\R)\qquad (n\geq 0)
\ee
whose kernel contains $H$ as an open subgroup. Equip on $G/H$ the aforementioned Nash structure. Then by Lemma \ref{subm0} and Lemma \ref{warner}, the map \eqref{vggln} descends to a submersive Nash map
\begin{equation}\label{ghgcov}
  G/H\rightarrow \psi(G).
\end{equation}
Since \eqref{ghgcov} is a group homomorphism as well as a finite fold covering map of topological spaces,
Proposition \ref{subg2} implies that $G/H$ is a Nash group, which is obviously almost linear.

\end{proof}

\section{Elliptic Nash groups}

We first observe that every compact subgroup of an  almost linear
Nash group is a Nash subgroup:

\begin{lem}\label{auts}
Let $G$ be an almost linear Nash group and let $K$ be a compact
subgroup of it. Then $K$ is a Nash subgroup of $G$.
\end{lem}
\begin{proof}
Fix a Nash homomorphism $\varphi: G\rightarrow \GL_n(\R)$ with
finite kernel. Write $K':=\varphi(K)$, which is a compact subgroup
of $\GL_n(\R)$. It is well know that $K'$ is semialgebraic in
$\GL_n(\R)$ (it is actually Zariski closed in $\GL_n(\R)$, cf.
\cite[Lemma 3.3.1]{Car}). Note that $\varphi^{-1}(K')$ is a Nash
subgroup of $G$, and has the same dimension as that of $K$.
Therefore $K$ is an open subgroup of $\varphi^{-1}(K')$, and is thus semialgebraic in
$G$.
\end{proof}

Recall from the Introduction that
an elliptic Nash group is defined to be an almost linear Nash group
which is compact as a topological space.

\begin{lem}\label{auts2}
Let $G$ be an almost linear Nash group and let $K$ be an elliptic Nash group. Then every Lie group homomorphism
$\varphi: K\rightarrow G$ is a Nash homomorphism. In particular,
every finite dimensional  real representation of $K$ is a Nash representation.
\end{lem}
\begin{proof}
The graph of $\varphi$ is a compact subgroup of the almost linear
Nash group $K\times G$. Therefore it is semialgebraic by Lemma
\ref{auts}.
\end{proof}

\begin{lem}\label{compactn}
Let $K$ be a compact Lie group. Then there is a unique Nash
structure on the underlying topological space of $K$ which makes $K$
an almost linear Nash group.
\end{lem}
\begin{proof}
Uniqueness follows from Lemma \ref{auts2}. To prove existence, fix
an injective Lie group homomorphism $\varphi: K\hookrightarrow
\GL_n(\R)$ (such a homomorphism always exists, cf. \cite[Section
3.3.C]{Car}). By Lemma \ref{auts}, $\varphi(K)$ is a Nash group. The existence then follows by transferring the Nash structure on $\varphi(K)$ to $K$,
through the topological group isomorphism $\varphi:
K\stackrel{\sim}{\rightarrow} \varphi(K)$.
\end{proof}

Combining Lemmas \ref{compactn}, \ref{auts2} and \ref{auts}, we get Theorem \ref{thmc} of the Introduction. Moreover, we have proved the following
\begin{thml}
The category of elliptic Nash groups is isomorphic to
the category of compact Lie groups.
\end{thml}

The following proposition is obvious.

\begin{prpl}\label{quotientel}
All Nash subgroups and Nash quotient groups of all elliptic Nash groups are elliptic as Nash groups.
\end{prpl}

Recall that a linear operator $x$ on a finite dimensional vector
space $V$  is said to be semisimple if every $x$-stable subspace of
$V$ has a complementary $x$-stable subspace. If $V$ is defined over
a field $\rk$ of characteristic zero, then for each field extension
$\rk'$ of $\rk$, $x$ is semisimple if and only if the $\rk'$-linear
operator
\[
  \rk'\otimes_\rk V \rightarrow \rk'\otimes_\rk V, \quad a\otimes v\mapsto a\otimes x(v)
\]
is semisimple (\cite[Chapter VII, Section 5.8]{Bo}). If $V$ is
defined over an algebraically closed field, then $x$ is semisimple
if and only if it is diagonalizable.

The following result concerning representations of compact Lie groups is well know. We provide a proof for completeness.

\begin{prpl}\label{je}
Let $V$ be a Nash representation of an elliptic Nash group $G$. Then each element of $G$ acts as a semisimple linear operator on $V$, and all its eigenvalues are complex numbers of modulus $1$.
\end{prpl}

\begin{proof}
Since every element of $G$ is contained in a compact abelian subgroup of $G$, we assume without loss of generality that  $G$ is abelian. Then the complexification $V_\C$ of $V$ is a direct sum of one dimensional subrepresentations. By choosing an appropriate basis of $V_\C$, the representation corresponds to a Nash homomorphism
 \begin{equation}\label{nashel}
   G\rightarrow (\C^\times)^n, \qquad  \textrm{where }n:=\dim V.
 \end{equation}
Compactness of $G$ implies that the image of \eqref{nashel} is contained in $\mathbb S^n$.
(Recall from the Introduction that $\mathbb S$ denotes the Nash group of complex numbers of modulus  one.) This proves the proposition.

\end{proof}

The following important result is due to Cartan,
Malcev and Iwasawa. For a proof, see  \cite[Theorem 1.2]{Borel} for
example.

\begin{thml}\label{compact}
Let $G$ be Lie group with finitely many connected components. Then
every compact subgroup of $G$ is contained in a maximal one, and all
maximal compact subgroups of $G$ are conjugate to each other.
Moreover, for each maximal compact subgroup $K$ of $G$, there exists a closed submanifold $X$ of $G$ which is diffeomorhpic to $\R^{\dim G-\dim K}$ such that the multiplication map $K\times X\rightarrow G$ is a diffeomorphism.
\end{thml}

\section{Unipotent Nash groups}

We say that a Nash group is unipotent if it has a
faithful Nash representation so that all group elements act as
unipotent linear operators. It is obvious that each Nash subgroup of a
unipotent Nash group is a unipotent Nash group.

First recall the following well know result, which is basic to the study of unipotent Nash groups.
\begin{lem}\label{nildiff} (cf. \cite{Dix})
For each connected, simply connected,  nilpotent Lie group $N$, the exponential map
\[
  \exp: \Lie \,N\rightarrow N
\]
is a diffeomorphism.
\end{lem}

Here and henceforth, ``$\Lie$" indicates the Lie algebra of a Lie group.

Recall from the Introduction that a subgroup of a Lie group $G$ is
said to be analytic if it equals the image of an injective Lie group
homomorphism from a connected Lie group to $G$.  Every analytic
subgroup is canonically a connected Lie group. We remark that in
general, the topology on a non-closed  analytic subgroup does not
coincide with the subspace topology. The set of all analytic
subgroups of $G$ is in one-one correspondence with  the set of all
Lie subalgebras of $\Lie\, G$.

\begin{lem}\label{nild}
Let $N$ be a  connected, simply connected, nilpotent Lie group. Then each analytic subgroup of $N$ is closed in $N$, and is a connected, simply connected, nilpotent Lie group.
\end{lem}
\begin{proof}
Let $\n_0$ be a Lie subalgebra of $\Lie\, N$. Let $N_0$ be a connected, simply connected Lie group, with Lie
algebra $\n_0$. Then $N_0$ is nilpotent, and there is a commutative diagram
\[
  \begin{CD}
            \n_0 @>\subset>> \Lie \, N \\
            @V\simeq V \exp  V           @V\simeq V\exp  V\\
           N_0@>\varphi>> N,\\
  \end{CD}
\]
where $\varphi$ is the Lie group homomorphism whose differential is the inclusion map $\n_0\hookrightarrow \Lie \,N$. By Lemma \ref{nildiff},  the two
vertical arrows are diffeomorphisms. Since the top
horizontal arrow is a closed embedding, $\varphi$ is also a closed
embedding. Then the lemma follows since $\varphi(N_0)$ is the analytic subgroup of $N$ corresponding to $\n_0$.
\end{proof}

\begin{lem}\label{unic3}
Let $V$ be a finite dimensional real representation of a connected Lie group $G$. If all elements of $G$ act as unipotent linear operators on $V$, then $G$ kills a full flag of $V$, namely, there exists a sequence
\[
  V_0\subset V_1\subset V_2\subset \cdots \subset V_n \qquad (n:=\dim V)
\]
of subspaces of $V$ such that $\dim V_i=i$ ($i=0,1,2,\cdots, n$), and
\[
  (g-1). V_i\subset V_{i-1}\quad \textrm{for all $g\in G$ and  all $i=1,2,\cdots, n$.}
\]
\end{lem}
\begin{proof}
Taking the differential of the representation, then $\Lie\, G$ acts as nilpotent linear operators on $V$. Therefore the Lemma is a direct consequence of Engel's Theorem.
\end{proof}

Now we come to the study of unipotent Nash groups.

\begin{lem}\label{unic2}
Every unipotent Nash group is connected.
\end{lem}
\begin{proof}
Proposition \ref{je} implies that a maximal compact subgroup of a unipotent Nash group is
trivial. Therefore the Lemma follows by Theorem \ref{compact}.
\end{proof}

Denote by $\oU_n(\R)$ ($n\geq 0$) the subgroup of $\GL_n(\R)$
consisting  all unipotent upper-triangular matrices. This is a unipotent
Nash group. As a Lie group, it is connected, simply connected, and nilpotent. The Lie algebra $\u_n(\R)$ of $\oU_n(\R)$ consisting all
nilpotent upper-triangular matrices in $\gl_n(\R)$.

\begin{lem}\label{cnil}
Every unipotent Nash group is Nash isomorphic to a
Nash subgroup of $\oU_n(\R)$, for some $n\geq 0$.
\end{lem}
\begin{proof}
This is a direct consequence of Lemmas \ref{unic2} and \ref{unic3}.
\end{proof}

\begin{prpl}\label{nashucs}
Every unipotent Nash group is connected, simply connected and nilpotent.
\end{prpl}
\begin{proof}
This is implied by Lemmas \ref{unic2}, \ref{cnil} and \ref{nild}.
\end{proof}

\begin{prpl}\label{autsu0}
For each unipotent Nash group $N$, the exponential map
\be \label{expunimap}
\exp: \Lie \, N\rightarrow N
\ee
is a Nash diffeomorphism.
\end{prpl}
\begin{proof}
Proposition \ref{nashucs} and Lemma \ref{nildiff} imply that \eqref{expunimap} is a diffeomorphism.  Using Lemma \ref{cnil}, fix an injective Nash homomorphism $\varphi: N\rightarrow \oU_n(\R)$.  Then we have a commutative diagram
\[
  \begin{CD}
            \Lie \,N @>\phi>> \u_n(\R) \\
            @V\simeq V \exp  V           @V\simeq V\exp  V\\
           N @>\varphi>> \oU_n(\R),\\
  \end{CD}
\]
where $\phi$ denotes the differential of $\varphi$. Note that in the above diagram, the right vertical arrow is a Nash diffeomorphism, and the two horizontal
arrows are Nash embeddings. Therefore the left vertical arrow is a Nash diffeomorphism.
\end{proof}

\begin{prpl}\label{nashucs000}
Every analytic subgroup of a unipotent Nash group is a Nash subgroup.
\end{prpl}
\begin{proof}
Let $N_0$ be an analytic subgroup of a unipotent Nash group $N$. Then we have a commutative diagram
\be \label{cd2}
  \begin{CD}
            \Lie \, N_0 @>\subset>> \Lie \,N \\
            @V \simeq V \exp  V           @V\simeq V\exp  V\\
           N_0@>\subset>> N.\\
  \end{CD}
\ee
By Proposition \ref{autsu0}, the right vertical arrow of \eqref{cd2} is a Nash diffeomorphism. By Lemmas \ref{nild} and \ref{nildiff}, the left vertical arrow of \eqref{cd2} is a diffeomorphism. Therefore $N_0$ is semialgebraic in $N$ since $\Lie \,N_0$ is semialgebraic in $\Lie \,N$.
\end{proof}

\begin{prpl}\label{autsu2}
Let $N$, $N'$ be two unipotent Nash groups. Then every Lie
group homomorphism $\varphi: N\rightarrow N'$ is a Nash
homomorphism.
\end{prpl}
\begin{proof}
The easily follows from Proposition \ref{autsu0}.
\end{proof}
\begin{prpl}\label{unin}
Let $N$ be a connected, simply connected,  nilpotent Lie group. Then
there exists a unique Nash structure on the underlying topological space
of $N$ which makes $N$ a unipotent Nash group.
\end{prpl}
\begin{proof}
The uniqueness assertion follows from Proposition \ref{autsu2}. To prove the existence, it suffices to show that there exists a unipotent Nash group which is isomorphic to $N$ as a Lie group. Since all unipotent Nash groups are connected and simply connected (Proposition \ref{nashucs}), it suffices to show that there exists a unipotent Nash group whose Lie algebra is isomorphic to $\Lie \,N$.

As a
special case of Ado's theorem, $\Lie \, N$ is isomorphic to a Lie
subalgebra of $\u_n(\R)$, for some $n\geq 0$ (cf. \cite[Theorem
7.19]{Mil}). Identify $\Lie \,N$ with a Lie subalgebra of $\u_n(\R)$. By Proposition \ref{nashucs000}, the corresponding analytic subgroup of $\oU_n(\R)$ is a unipotent Nash group. Therefore the proposition follows.
\end{proof}

Combining Propositions \ref{nashucs}, \ref{autsu2}, \ref{unin} and \ref{nashucs000}, we get Theorem \ref{thmu} of the Introduction. We have also proved the following
\begin{thml}
The category of unipotent  almost linear Nash groups is isomorphic to the category of  connected, simply connected, nilpotent
Lie groups.
\end{thml}

The following lemma will be used later.
\begin{lem}\label{nilc}
Let $N$ be a connected Nash group. If $N$ has a Nash representation with
finite kernel so that all group elements act as unipotent linear
operators, then $N$ is unipotent.
\end{lem}
\begin{proof}
Using Lemma \ref{unic3}, we get a Nash
homomorphism $\varphi: N \rightarrow \oU_n(\R)$ with finite kernel.
By Proposition \ref{nashucs}, $\varphi(N)$ is connected and simply
connected. This implies that $\varphi: N\rightarrow \varphi(N)$ is a
Nash isomorphism, and the lemma follows.

\end{proof}

\section{Hyperbolic Nash groups}
The group $\R^\times_+$ of positive real numbers is an almost linear
Nash group in the obvious way. We define a hyperbolic Nash group to be a Nash group which is Nash isomorphic to $(\R^\times_+)^n$
for some $n\geq 0$.

As in the Introduction, for every two Nash groups $G$ and $G'$, write $\Hom(G,G')$ for the
set of all Nash homomorphisms from $G$ to $G'$. It is obviously an
abelian group when $G'$ is abelian, and is a ring when $G=G'$ and $G'$ is abelian.

We leave the proof of the following lemma to the interested reader.
 \begin{lem}\label{homs}
 The map
 \begin{equation}\label{endr}
   \mathbb Q\rightarrow \Hom(\R^\times_+,\R^\times_+),\qquad r\mapsto(x\mapsto x^r)
 \end{equation}
 is a ring isomorphism.
 \end{lem}

Using the isomorphism \eqref{endr}, we view $\R_+^\times$ as a right $\mathbb Q$-vector space as in the Introduction. Then for every hyperbolic Nash group $A$, the ablian group $\Hom(\R^\times_+,
A)$ is a left $\mathbb Q$-vector space:
\[
  r\cdot \varphi:=\varphi\circ (\cdot)^r, \qquad r\in \mathbb Q, \,\varphi\in  \Hom(\R^\times_+, A),
\]
where $(\cdot)^r$ denotes the endomorphism $x\mapsto x^r$ of
$\R^\times_+$. By Lemma \ref{homs}, the dimension of
$\Hom(\R^\times_+, A)$ equals that of $A$.

On the other hand, given a $\mathbb Q$-vector space $E$ of finite
dimension $k$, the tensor product
\[
  \R^\times_+\otimes_\mathbb Q E
\]
is obviously a hyperbolic Nash group of dimension $k$.

The following theorem repeats Theorem \ref{thmh} of the Introduction.
\begin{thml}\label{equsplit}
The functor
\[
  A\mapsto \Hom(\R^\times_+, A)
\]
establishes an equivalence from the category of hyperbolic Nash groups to the category of finite dimensional left $\mathbb Q$-vector spaces. It
has a quasi-inverse
\[
  E\mapsto \R^\times_+\otimes_\mathbb Q E.
\]
\end{thml}
\begin{proof}
This is an obvious consequence of Lemma \ref{homs}.
\end{proof}

Moreover, we have the following
\begin{prpl}\label{subsplit}
Every Nash subgroup of a hyperbolic Nash group is a hyperbolic Nash group.
\end{prpl}
\begin{proof}
 Let $H$ be a Nash subgroup of $A_1\times A_2\times \cdots\times A_k$, where $k\geq 0$, and each $A_i$ is a Nash group which is Nash isomorphic to $\R_+^\times$ ($i=1,2,\cdots, k$).
 We want to show that $H$ is a hyperbolic Nash group. Assume without loss of generality that $H\neq A$. Then $A_i$ is not contained in $H$ for some $i$.
 Then $A_i\cap H=\{1\}$, and we get an injective Nash homomorphism $H\hookrightarrow A/A_i$. The lemma then follows by an inductive argument.
\end{proof}

Likewise, we have
\begin{prpl}\label{quosplit}
Every Nash quotient group of a hyperbolic Nash group is a hyperbolic Nash group.
\end{prpl}
\bp
Let $A_0$ be a Nash subgroup of a hyperbolic Nash group $A$. By Proposition
\ref{subsplit}, $A_0$ is also a hyperbolic Nash group. Using Theorem \ref{equsplit}, we get an exact sequence
\[
  0\rightarrow \Hom(\R_+^\times, A_0)\rightarrow   \Hom(\R_+^\times, A)\rightarrow  \Hom(\R_+^\times, A)/ \Hom(\R_+^\times, A_0)\rightarrow 0
\]
of left $\mathbb Q$-vector spaces. Tensoring with  $\R_+^\times$, we get an exact sequence
\[
  1\rightarrow A_0\rightarrow A\rightarrow \R_+^\times \otimes_\mathbb Q ( \Hom(\R_+^\times, A)/ \Hom(\R_+^\times, A_0))\rightarrow 1
\]
of hyperbolic Nash groups and Nash homomorphisms. Therefore the proposition follows.
\ep

\section{Disjointness of Elliptic, hyperbolic and unipotent Nash groups}\label{secjordan}

First recall the following well known fact:
\begin{lem}\label{j00}
The hyperbolic Nash group $\R^\times_+$ is not Nash isomorphic to the
unipotent Nash group $\R$.
\end{lem}
\begin{proof}
 Note that the Nash endomorphism ring $\Hom(\R,\R)$ is isomorphic to $\R$. By Lemma \ref{homs}, the  Nash endomorphism ring $\Hom(\R_+^\times,\R_+^\times)$ is isomorphic to $\mathbb Q$. Therefore the lemma holds.
\end{proof}

The following is a useful fact about unipotent Nash groups:
\begin{lem}\label{euni}
Every non-trivial element of a unipotent Nash group is contained in a Nash subgroup which is Nash isomorphic to $\R$.
\end{lem}
\bp
This is directly implied by Proposition \ref{autsu0}.
\ep

Elliptic Nash groups, hyperbolic Nash groups and unipotent Nash groups are disjoint to each other in the following sense:

\begin{prpl}\label{jt}
Let $G$ and $G'$ be two Nash groups. If $G$ is elliptic,
and $G'$ is hyperbolic or unipotent, then
\[
  \Hom(G,G')=\{1\}.
\]
The same holds if $G$ is hyperbolic, and
$G'$ is elliptic or unipotent; or if $G$ is unipotent, and $G'$ is
elliptic or hyperbolic.
\end{prpl}
\begin{proof}
Note that all hyperbolic Nash groups and all unipotent Nash groups have no non-trivial compact subgroups. Therefore, $\Hom(G,G')=\{1\}$ if $G$ if elliptic, and $G'$ is hyperbolic or unipotent.

Note that if $G'$ is elliptic, then
\be \label{homrr}
  \Hom(\R_+^\times, G')=\{1\}\quad\textrm{and}\quad \Hom(\R, G')=\{1\}.
\ee
The first equality of \eqref{homrr} implies that $\Hom(G, G')=\{1\}$ if $G$ is hyperbolic and $G'$ is elliptic. By Lemma \ref{euni}, the second equality of  \eqref{homrr} implies that $\Hom(G, G')=\{1\}$ if $G$ is unipotent and $G'$ is elliptic.

Lemma \ref{j00} implies that
\[
  \Hom(\R_+^\times, G')=\{1\}
\]
if $G'$ is unipotent. Therefore $\Hom(G, G')=\{1\}$ if $G$ is hyperbolic and $G'$ is unipotent. Lemma \ref{j00} also implies that
\be \label{homrr3}
  \Hom(\R, G')=\{1\}
\ee
if $G'$ is hyperbolic. By Lemma \ref{euni}, \eqref{homrr3} implies that $\Hom(G, G')=\{1\}$ if $G$ is unipotent and $G'$ is hyperbolic. This finishes the proof of the proposition.
\end{proof}

\begin{prpl}\label{intehu}
Let $H_1$, $H_2$, $H_3$ be three Nash subgroups of a Nash group $G$. If they are respectively elliptic, hyperbolic and unipotent, then they have pairwise trivial intersections.
\end{prpl}
\begin{proof}
The Nash group $H_1\cap H_2$ is elliptic and hyperbolic, and is hence trivial by Proposition \ref{jt}. Likewise, $H_1\cap H_3$ and $H_2\cap H_3$ are trivial.
\end{proof}

\begin{lem}\label{inteh0}
Let $G_1$ be an elliptic Nash group and let $G_2$ ba a hyperbolic Nash group. Then all Nash subgroups of $G_1\times G_2$ are of the form $H_1\times H_2$, where $H_i$ is a Nash subgroup of $G_i$, $i=1,2$.
\end{lem}
\begin{proof}
Let $H$ be a Nash subgroup of $G_1\times G_2$. We first claim that
\begin{equation}\label{inteh}
\textrm{if $\quad H\cap G_2=\{1\},\quad$ then $\quad H\subset G_1$.}
\end{equation}
Consider the restriction to $H$ of the projection map
\[
  G_1\times G_2\rightarrow G_1.
  \]
The condition $H\cap G_2=\{1\}$ implies that $H$ is an elliptic Nash group.
Then by Proposition \ref{jt}, the projection map
\[
G_1\times G_2\rightarrow G_2
\]
has trivial restriction to $H$. Therefore $H\subset G_1$, and the claim is proved.

In general, put $G_2':=G_2/(G_2\cap H)$, which is a hyperbolic Nash group.
Write $H'$ for the image of $H$ under the Nash homomorphism
\[
  p: G_1\times G_2\rightarrow G_1\times G_2',\quad (g_1,g_2)\mapsto (g_1, g_2(G_2\cap H)).
\]
Inside the group $G_1\times G_2'$, we have that
\[
  H'\cap G_2'=\{1\},
 \]
and then \eqref{inteh} implies that
\[
H'\subset G_1.
\]
The lemma then follows as $H=p^{-1}(H')$.

\end{proof}

\begin{lem}\label{intnuu}
Let $G_1$ be the direct product of an elliptic Nash group and a hyperbolic Nash group, and let $G_2$ ba a unipotent Nash group.
Then all Nash subgroups of $G_1\times G_2$ are of the form $H_1\times H_2$, where $H_i$ is a Nash subgroup of $G_i$, $i=1,2$.

\end{lem}

\begin{proof}
Lemma \ref{inteh0} and its proof show the lemma when $G_2$ is abelian. In general, let $H$ be a Nash subgroup of $G_1\times G_2$, and let $xy\in H$, where $x\in G_1$ and $y\in G_2$. It suffices to show that $y\in G_2$.
Replacing  $G_2$ by an abelian Nash subgroup $G_2'$ of $G$ containing $y$, and replacing  $H$ by $H\cap (G_1\times G_2')$, the lemma is reduced to the case when $G_2$ is abelian.
\end{proof}

Combining Lemmas \ref{inteh0} and \ref{intnuu}, we get
\begin{prpl}\label{intehu2}
Let $G_1$, $G_2$, $G_3$ be three Nash groups which are respectively elliptic, hyperbolic and unipotent. Then every Nash subgroup of $G_1\times G_2\times G_3$ is of the form $H_1\times H_2\times H_3$,
where $H_i$ is a Nash subgroup of $G_i$, $i=1,2,3$.
\end{prpl}

As a direct consequence of Proposition \ref{intehu2}, we have
\begin{prpl}\label{intehu3}
Let $H_1$, $H_2$, $H_3$ be three Nash subgroups of a Nash group $G$ which are respectively elliptic, hyperbolic and unipotent. If they pairwise commute with each other, then the multiplication map $H_1\times H_2\times H_3\rightarrow G$ is an injective Nash homomorphism.
\end{prpl}

In the rest of this section, we draw some consequences of Proposition \ref{jt} on unipotent Nash groups and hyperbolic Nash groups.

\begin{prpl}\label{j0unip}
Let $V$ be a Nash representation of a unipotent Nash group $G$. Then each element of $G$ acts as a unipotent linear operator on $V$.
\end{prpl}
\begin{proof}
Using Lemma \ref{euni}, we assume without loss of generality that $G=\R$. Let $V_1$ be an irreducible subquotient representatoin of the complexification $V_\C$ of $V$. Since $G$ is abelian, it is one dimensional and  corresponds to a Nash homomorphism
\[
  G\rightarrow \C^\times.
\]
This homomorphism is trivial by Proposition \ref{jt}. Therefore the proposition follows.
\end{proof}

As a consequence of Proposition \ref{j0unip}, we have
\begin{prpl}\label{quotientunip}
Every Nash quotient group of a unipotent Nash group is
unipotent.
\end{prpl}

\begin{proof}
Let $N$ be a unipotent Nash group and let $N'$ be a Nash quotient group of it. Fix a Nash representation $V$ of $N'$ with finite
kernel. Applying Proposition \ref{j0unip} to the inflation of the representation $V$ to $N$, we know that $N'$ acts on $V$ as unipotent linear operators.
Then the proposition follows by Lemma \ref{nilc}.
\end{proof}

\begin{prpl}\label{t}
All irreducible Nash representations of all unipotent Nash groups are
trial.
\end{prpl}
\begin{proof}
This is implied by Proposition \ref{j0unip} and Lemma \ref{unic3}.
\end{proof}

Now we consider Nash representations of hyperbolic Nash groups.

\begin{lem}\label{juhyp}
Let $V$ be a Nash representation of a hyperbolic Nash group $G$. If each element of $G$ acts as a unipotent linear operator on $V$, then the representation $V$ is trivial.
\end{lem}
\begin{proof}
By Lemma \ref{unic3}, the image of the attached homomorphism $G\rightarrow \GL(V)$ is contained in a unipotent Nash subgroup of $\GL(V)$. Therefore the homomorphism is trivial by Proposition \ref{jt}.
\end{proof}

\begin{prpl}\label{jhyp}
Let $V$ be a Nash representation of a hyperbolic Nash group $G$. Then each element of $G$ acts as a semisimple linear operator on $V$, and all its eigenvalues are positive real numbers.
\end{prpl}

\begin{proof}
By Proposition \ref{jt}, the image of every Nash homomorphism from $G$ to
$\C^\times$ is contained in $\R_+^\times$. This implies that for
every $g\in G$, all eigenvalues of $\varphi(g)$ are positive real
numbers, where $\varphi: G\rightarrow \GL(V)$ denotes the Nash
homomorphism attached to the representation. Using the generalized
eigenspace decomposition, we assume without loss of generality that
there is a Nash homomorphism $\chi: G\rightarrow \R_+^\times$ such
that for every $g\in G$, all eigenvalues of $\varphi(g)$ are equal to
$\chi(g)$. Then $G$ acts on $V\otimes \chi^{-1}$ by unipotent linear operators.  This action is trivial by Lemma \ref{juhyp}. Therefore $G$ acts on
$V$ via the character $\chi$, and the proposition is proved.
\end{proof}

Proposition \ref{jhyp} clearly implies the following

\begin{prpl}\label{jhyp2}
Every Nash representation of a hyperbolic Nash group is a direct sum of one dimensional subrepresentations.
\end{prpl}

\section{Jordan decompositions}\label{secjordan}

Let $G$ be an almost linear Nash group throughout this section. For every $x\in G$, define
its replica $\la x \ra$ to be the smallest Nash subgroup of $G$
containing $x$. This is well defined by Proposition \ref{subgi}. It is easy to see that $\la x\ra$ is abelian.

We say that $x\in G$ is
elliptic, hyperbolic  or unipotent if it is contained in a Nash subgroup of $G$ which is elliptic, hyperbolic or unipotent, respectively. This is equivalent to saying that the abelian Nash group $\la x \ra$ is respectively elliptic, hyperbolic or unipotent.  Respectively write
$G_\mathrm e$,  $G_\mathrm h$ and  $G_\mathrm u$ for the sets of all
elliptic, hyperbolic and unipotent elements in $G$.

\begin{lem}\label{jgl}
An element in $\GL_n(\R)$ ($n\geq 0$) is elliptic if and only if it is
semisimple and all its eigenvalues are complex numbers of modulus
one; it is hyperbolic if and only if it is semisimple and all its
eigenvalues are positive real numbers; it is unipotent if and only
if all its eigenvalues are equal to $1$.
\end{lem}
\begin{proof}
The ``if" parts of the three assertions of the lemma are obvious. The ``only if" parts are implied by Propositions \ref{je}, \ref{jhyp} and  \ref{j0unip}.

\end{proof}

\begin{lem}\label{j0}
Let $e, h, u\in G$. Assume that they are respectively elliptic, hyperbolic and unipotent, and they pairwise
commute with each other.  Then
\[
\la e h u\ra\supset \la e\ra, \,\la
h \ra,\, \la u\ra,
\]
and the multiplication map
\[
\la e\ra\times\la h
\ra\times \la u\ra\rightarrow \la e h u\ra
\]
is an isomorphism of
Nash groups.
\end{lem}
\begin{proof}
First note that the subgroups $\la e\ra, \,\la h \ra,\, \la u\ra$
are pairwise commutative to each other. Using Proposition \ref{intehu3}, we view $\la e\ra\times\la h
\ra\times \la u\ra$ as a Nash subgroup $G$. Then $\la ehu\ra$ is a Nash subgroup of $\la e\ra\times\la h
\ra\times \la u\ra$, and Proposition \ref{intehu2} implies that
\[
  \la ehu\ra=\la e\ra\times\la h
\ra\times \la u\ra.
\]
This proves the lemma.
\end{proof}

Here is the Jordan decomposition theorem for almost linear Nash groups:
\begin{thml}\label{jd}
Every element $x$ of an almost linear Nash group $G$ is uniquely of
the form $x=e h u$ such that $e\in G_\mathrm e$, $h\in G_\mathrm h$, $u\in G_\mathrm u$, and they pairwise commute with each other.
\end{thml}
\begin{proof}
Fix a Nash homomorphism $\varphi: G\rightarrow \GL_n(\R)$  with
finite kernel. Put $y:=\varphi(x)$ and write $y=y_e y_h y_u$ for the
usual Jordan decomposition of $y$ in $\GL_n(\R)$, where $y_e$ is elliptic, $y_h$
is hyperbolic, $y_u$ is unipotent and  they pairwise commute with each
other (\cite[Page 430-431]{He}). Then Lemma \ref{j0} implies that
\[
y_e, y_h, y_u\in  \varphi(G).
\]
Denote by $h$ the unique element in
the identity connected component of $\varphi^{-1}(\la y_h\ra)$ which
lifts $y_h$. Define $u$ similarly, and put $e:=x u^{-1} h^{-1}$.
Then it is routine to check that $(e, h, u)$ is the unique triple
which fulfills all the requirements of the theorem.
\end{proof}

The equality $x=ehu$ of Theorem \ref{jd} is called the Jordan
decomposition of $x\in G$. We respectively use $x_\mathrm e$, $x_\mathrm h$ and $x_\mathrm u$ to denote the elements $e$, $h$ and $u$. They are respectively called the elliptic, hyperbolic and unipotent parts of $x\in G$.

\begin{prpl}\label{pr}
Let $\varphi: G\rightarrow G'$ be a Nash homomorphism of almost
linear Nash groups. Then
\be \label{incjd}
  \varphi(G_\mathrm e)\subset G'_\mathrm e, \quad \varphi(G_\mathrm h)\subset G'_\mathrm h\quad \textrm{and}\quad \varphi(G_\mathrm u)\subset G'_\mathrm u.
\ee If $\varphi$ is surjective, then the three inclusions in
\eqref{incjd} become equalities.
\end{prpl}
\begin{proof}
The three inclusions are respectively implied by Propositions
\ref{quotientel}, \ref{quosplit} and \ref{quotientunip}.

Now assume that $\varphi$ is surjective, and let $y\in
G'_\mathrm{e}$. Pick $x\in G$ so that $\varphi(x)=y$. Then
$\varphi(x_\mathrm{e}) \varphi(x_\mathrm{h}) \varphi(x_\mathrm
u)=y$. By \eqref{incjd}, uniqueness of Jordan decompositions implies that
$\varphi(x_\mathrm e)=y$. This proves that $\varphi(G_\mathrm e)=
G'_\mathrm e$. The same argument proves the other two equalities.
\end{proof}

Proposition \ref{pr} obviously implies that Nash homomorphisms preserve
Jordan decompositions:
\begin{prpl}\label{pjd}
 Let $\varphi: G\rightarrow G'$ be a Nash homomorphism of almost linear Nash groups. Then for every $x\in G$, one has that
\[
  (\varphi(x))_\mathrm e=\varphi(x_\mathrm e),\quad  (\varphi(x))_\mathrm h=\varphi(x_\mathrm h)\quad\textrm{and} \quad (\varphi(x))_\mathrm u=\varphi(x_\mathrm u).
\]
\end{prpl}

As one application of Jordan decompositions, we get the following result about  structures of  abelian  almost linear Nash groups:

\begin{prpl}\label{da}
Let $G$ be an abelian almost linear Nash group. Then $G_\mathrm e$ is an elliptic Nash subgroup of $G$,
$G_\mathrm h$ is a hyperbolic Nash subgroup of $G$, and $G_\mathrm u$ is a unipotent Nash subgroups of $G$.
Moreover, the multiplication map
\[
  G_\mathrm e\times G_\mathrm h\times G_\mathrm u\rightarrow G
  \]
 is a Nash isomorphism.
\end{prpl}
\begin{proof}
Let $K$ be a maximal compact subgroup of $G$, which is unique since
$G$ is abelian. Then clearly $K=G_\mathrm e$. Let $A$ be a hyperbolic Nash subgroup of $G$ of  maximal dimension.  Then clearly $A=G_\mathrm h$. Likewise,
let $U$ be a unipotent Nash subgroup of $G$ of  maximal dimension.
Then  $U=G_\mathrm u$.  The last assertion follows from Theorem
\ref{jd}.
\end{proof}

In the rest of this section, denote by $\g$ the Lie algebra of the almost linear Nash group $G$. For
every $x\in \g$, we define its replica $\la x\ra$ to be the smallest
Nash subgroup of $G$ containing $\exp(\R x)$. It is connected and abelian. We said that $x\in \g$ is  elliptic, hyperbolic, or unipotent if the Nash group
$\la x \ra$ is respectively elliptic, hyperbolic, or unipotent.  As in the group case, respectively write
$\g_\mathrm e$,  $\g_\mathrm h$ and  $\g_\mathrm u$ for the sets of all
elliptic, hyperbolic and unipotent elements in $\g$.

Lemma \ref{jgl} easily implies the following

\begin{lem}\label{jgl2}
An element in the Lie algebra $\gl_n(\R)$ of $\GL_n(\R)$ ($n\geq 0$) is  elliptic if and only if it is semisimple and all its
eigenvalues are purely imaginary; it is hyperbolic if and only if it is
semisimple and all its eigenvalues are real; it is unipotent if and
only if it is a nilpotent matrix.
\end{lem}

The same proof as Lemma \ref{j0} shows the following
\begin{lem}\label{j02}
Let $e, h, u\in \g$. Assume that they are respectively elliptic, hyperbolic and unipotent, and they pairwise
commute with each other. Then
\[
\la e+h+u\ra\supset \la e\ra, \,\la
h \ra,\, \la u\ra,
\]
and the multiplication map
\[
 \la e\ra\times\la h
\ra\times \la u\ra\rightarrow \la e +h+ u\ra
\]
is an isomorphism of Nash groups.
\end{lem}

The following is the Jordan decomposition theorem at Lie algebra level:
\begin{thml}\label{jdi}
Let $G$ be an almost linear Nash group with Lie algebra $\g$. Then every element $x\in \g$ is uniquely of the form $x=e+h+u$ such that
$e\in \g_\mathrm e$, $h\in \g_\mathrm h$, $u\in \g_\mathrm u$, and they pairwise
commute with each other.
\end{thml}
\bp
The proof is similar to that of Theorem  \ref{jd}. We omit the details.

\ep

We also call the equality $x=e+h+u$ of Theorem \ref{jdi} the Jordan
decomposition of $x\in \g$. As in the group case, we respectively use $x_\mathrm e$, $x_\mathrm h$ and $x_\mathrm u$ to denote the elements $e$, $h$ and $u$. They are respectively called the elliptic, hyperbolic and unipotent parts of $x\in\g$.

The same proof as Proposition \ref{pr} shows the following
\begin{prpl}\label{pr2}
Let $\varphi: G\rightarrow G'$ be a Nash homomorphism of almost
linear Nash groups. Write  $\phi: \g\rightarrow \g'$ for its
differential, where $\g':=\Lie \,G'$. Then \be \label{incjdalg}
  \phi(\g_\mathrm e)\subset \g_\mathrm e, \quad \phi(\g_\mathrm h)\subset \g'_\mathrm h\quad \textrm{and}\quad \phi(\g_\mathrm u)\subset \g'_\mathrm u.
\ee If $\phi$ is surjective, then the three inclusions in
\eqref{incjdalg} become equalities.
\end{prpl}

Similar to Proposition \ref{pjd}, the above proposition implies the following

\begin{prpl}\label{pr3}
Let $\varphi: G\rightarrow G'$ be a Nash homomorphism of almost
linear Nash groups. Write  $\phi: \g\rightarrow \g'$ for its
differential,  where $\g':=\Lie \,G'$. Then for every $x\in \g$,
one has that
\[
  (\phi(x))_\mathrm e=\phi(x_\mathrm e),\quad  (\phi(x))_\mathrm h=\phi(x_\mathrm h)\quad\textrm{and} \quad (\phi(x))_\mathrm u=\phi(x_\mathrm u).
\]
\end{prpl}

\section{Exponential elements}

Let $G$ be an almost linear Nash group with Lie algebra $\g$. The
following lemma concerning the exponential map is obvious.
\begin{lem}\label{pr1}
One has that
\[
\exp(\g_\mathrm e)\subset G_\mathrm e,\quad
 \exp(\g_\mathrm h)\subset G_\mathrm h\quad \textrm{and}\quad \exp(\g_\mathrm
u)\subset G_\mathrm u.
\]
\end{lem}
For each $x\in G_\mathrm h$ or $G_\mathrm u$, define $\log (x)$ to be the
unique element in the Lie algebra of $\la x\ra$ such that $\exp(\log (x))=x$. Then $\log (x)$ belongs to $\g_\mathrm h$ or $\g_\mathrm
u$, respectively. The maps
\[
  \exp: \g_\mathrm h\rightarrow G_\mathrm h\quad \textrm{ and }\quad
  \log:  G_\mathrm h\rightarrow \g_\mathrm h
\]
are inverse to each other. Likewise, the maps
\[
  \exp: \g_\mathrm u\rightarrow G_\mathrm u\quad \textrm{ and }\quad
  \log:  G_\mathrm u\rightarrow \g_\mathrm u
\]
are inverse to each other.

\begin{lem}\label{exphu}
Let $x\in \g_\mathrm h$ or $\g_\mathrm u$. Then $\la x\ra=\la \exp(x)\ra$.
\end{lem}
\begin{proof}
The Nash subgroup  $\la x \ra$ of $G$ contains the Nash subgroup $\la \exp (x)\ra$.
Since $x\in \Lie\, \la x\ra$, and $\exp (x)\in \la \exp (x) \ra$, using the  commutative diagram
\[
  \begin{CD}
            \Lie\, \la \exp (x)\ra @>\subset>>  \Lie\, \la x\ra \\
            @V\simeq V \exp  V           @V\simeq V\exp  V\\
           \la \exp (x)\ra @>\subset>> \la x\ra,\\
  \end{CD}
\]
we know that $x\in \Lie\, \la \exp (x)\ra$. Therefore
\[
\la \exp (x)\ra \supset \exp (\R x),
\]
and hence $\la \exp (x)\ra=\la x\ra $.

\end{proof}

\begin{lem}\label{expcom}
Let $x\in G_\mathrm h$ and let $y\in G_\mathrm u$. If they commute with each other in $G$, then $\log (x)$ and $\log (y)$ commute with each other in $\g$.
\end{lem}
\begin{proof}
If $x$ and $y$ commute with each other, then Lemma \ref{exphu} implies that the Nash subgroups  $\la \log (x)\ra$ and  $\la \log (y)\ra$ commute with each other. Therefore, $\la \log (x)\ra \,\la \log (y)\ra$ is a abelian Nash subgroup of $G$. Then the lemma follows since both $\log (x)$ and $\log (y)$ belong to the Lie algebra of $\la \log (x)\ra \,\la \log (y)\ra$.

\end{proof}

\begin{dfnl}
An element of an almost linear Nash group $G$ or its Lie algebra $\g$ is said to be exponential if its elliptic
part is trivial.
\end{dfnl}
 Denote by $G_\mathrm{ex}$
and $\g_\mathrm{ex}$ the sets of all exponential elements in $G$ and $\g$,
respectively. For every exponential element $x\in G_\mathrm{ex}$,
define
\[
  \log (x):=\log (x_\mathrm h) +\log (x_\mathrm u).
\]
By Lemma \ref{expcom}, this is an element of $\g_\mathrm{ex}$

\begin{prpl}\label{expp0}
The maps
\[
  \exp: \g_\mathrm{ex}\rightarrow G_\mathrm{ex}\quad \textrm{ and }\quad
  \log:  G_\mathrm{ex}\rightarrow \g_\mathrm{ex}
\]
are inverse to each other.
\end{prpl}
\begin{proof}
This is obvious.
\end{proof}

\begin{prpl}\label{expp}
Let $x\in \g_\mathrm{ex}$. Then $\la x\ra=\la \exp (x)\ra$.
\end{prpl}
\begin{proof}
 By Lemmas \ref{exphu}, \ref{j0} and \ref{j02}, we have that
\[
  \la \exp (x)\ra=\la \exp (x_\mathrm h)\ra \times \la \exp (x_\mathrm u) \ra=\la x_\mathrm h\ra \times \la x_\mathrm u\ra=\la x\ra.
\]
\end{proof}

\section{Semisimple Nash groups}

We say that a Lie group (or a Nash group) is semisimple if its Lie
algebra is semisimple.

\begin{lem}\label{ssl}
Every semisimple Nash group is almost linear.
\end{lem}
\begin{proof}
Taking the adjoint representation, then the lemma follows.
\end{proof}

Recall the following
\begin{lem}\label{scenter}(\cite[Proposition 7.9]{Kn})
Every semisimple analytic subgroup of $\GL_n(\R)$ ($n\geq 0$)  has a
finite center.
\end{lem}

Using Cartan decompositions for semisimple Lie groups (cf. \cite[Theorem 7.39]{Kn}), we easily
get the following
\begin{lem}\label{scenter2}
Let $G$ be a connected semisimple Lie group with finite center. Let
$K$ be a maximal compact subgroup of $G$. Then there are analytic
subgroups $H_1, H_2, \cdots, H_r$ ($r\geq 0$) of $G$ such that
\begin{itemize}
  \item  $G=K H_1 H_2\cdots H_r K$; and
  \item  the Lie algebra of $H_i$ is isomorphic to $\sl_2(\R)$ ($i=1,2,\cdots, r$).
\end{itemize}
\end{lem}

Recall that every analytic subgroup of $\GL_n(\R)$ is isomorphic to  either $\SL_2(\R)$ or
$\SL_2(\R)/\{\pm 1\}$ if its Lie algebra is isomorphic to $\sl_2(\R)$. Representation theory of $\sl_2(\R)$ implies the following

\begin{lem}\label{scenter3}
Every finite dimensional real representation of $\SL_2(\R)$ or
$\SL_2(\R)/\{\pm 1\}$ is a Nash representation.
\end{lem}

As a direct consequence of Lemma \ref{scenter3}, we have

\begin{lem}\label{scenter4}
An analytic subgroup of $\GL_n(\R)$ is a Nash subgroup if its
Lie algebra is isomorphic to $\sl_2(\R)$.
\end{lem}

Combining Lemmas  \ref{scenter}, \ref{scenter2}, \ref{scenter4} and
\ref{auts}, we have

\begin{lem}\label{scenter5}
Every semisimple analytic subgroup of $\GL_n(\R)$ ($n\geq 0$) is a
Nash subgroup.
\end{lem}

By Lemma \ref{scenter5}, the same proof as that of Lemma \ref{auts}
implies the following

\begin{prpl}\label{autss}
Every semisimple analytic subgroup of every almost linear Nash group
is a Nash subgroup.
\end{prpl}

Similar to the proof of Lemma \ref{auts2}, Proposition \ref{autss}
implies the following
\begin{prpl}\label{autss2}
Every Lie group homomorphism from a semisimple Nash group to an
almost linear Nash group is a Nash homomorphism. In particular,
every finite dimensional representation of a semisimple Nash group
is a Nash representation.
\end{prpl}

Each semisimple Nash group has finitely many
connected components, and its identity
connected component has a finite center. Conversely, we have
\begin{prpl}\label{snash}
Let $G$ be a semisimple Lie group. If it has finitely many connected
components, and its identity connected component has a finite
center, then there exists a unique Nash structure on the underlying
topological space of $G$ which makes $G$ a Nash group.
\end{prpl}
\begin{proof}
Denote by $\g$ the Lie algebra of $G$. The automorphism group $\Aut(\g)$ of $\g$ is obviously a Nash group. The adjoint representation $\Ad: G\rightarrow \Aut(\g)$ has open
image and finite kernel. Therefore the existence follows by
Propositions \ref{subg2}. The uniqueness is implied by Proposition
\ref{autss2}.
\end{proof}

In conclusion, we have proved the following
\begin{thml}
The category of semisimple Nash groups is isomorphic to the category
of semisimple Lie groups which have finitely many connected
components, and whose identity connected components have a finite
center.
\end{thml}

Recall the following famous result of Weyl:
\begin{lem}\label{ssg0} (cf. \cite[Theorem 1]{Jac})
Let $\g$ be a semisimple finite dimensional Lie algebra over a field
$\rk$ of characteristic zero. Then all of its finite dimensional
representations over $\rk$ are completely reducible.
\end{lem}

Also recall the following elementary lemma:

\begin{lem}\label{g12}(cf. \cite[Lemma 3.1]{Mo})
Let $H$ be a normal subgroup of a group $G$. Let
$V$ be a representation of $G$ over a field $\rk$.
\begin{itemize}
    \item If $V$ is finite dimensional and completely reducible, then its  restriction to $H$ is completely reducible.
    \item Assume that $H$ has finite index in $G$, and $\rk$ has characteristic zero. Then $V$ is completely reducible if its restriction to $H$ is so.
  \end{itemize}
\end{lem}

Combining Lemma \ref{ssg0} and the second assertion of Lemma \ref{g12}, we get
\begin{lem}\label{ssg}
Every Nash representation of a semisimple Nash group is completely
reducible.
\end{lem}

\section{Reductive Nash groups}

We say that a Nash group is reductive if it has a completely
reducible Nash representation with finite kernel. Using induced
representations  as in the proof of Proposition \ref{induction},
Lemma \ref{g12} easily implies the following
\begin{lemt}\label{redi}
A Nash group is reductive if and only if its identity connected
component is reductive.
\end{lemt}

Recall from the Introduction that  a Nash torus is a Nash group which is Nash isomorphic to $\mathbb S^m\times (\R_+^\times)^n$ for some $m,n\geq 0$. Every Nash torus is clearly a reductive Nash group. The main result we will prove in this section is the following

\begin{thm}\label{strr}
A connected Nash group $G$ is reductive if and only if there exist a
connected semisimple Nash group $H$,  a Nash torus $T$, and a surjective Nash homomorphism $H\times T\rightarrow G$ with finite kernel .
\end{thm}

Recall that a finite dimensional Lie algebra is said to be reductive
if its adjoint representation is completely reducible, or
equivalently, if it is the direct sum of an abelian Lie algebra and
a semisimple Lie algebra. Recall the following result of N.
Jacobson:

\begin{lemt}\label{jacob} (\cite[Theorem 1]{Jac})
Let $\g$ be a  finite dimensional
Lie algebra over a field $\rk$ of characteristic zero. If $\g$ has a faithful completely reducible finite
dimensional representation over $\rk$, then $\g$ is reductive.
\end{lemt}

Combining Lemmas \ref{redi} and \ref{jacob}, we get
\begin{lemt}\label{jacob2}
The Lie algebra of every reductive almost linear Nash group is
reductive.
\end{lemt}

Let $G$ be a connected reductive Nash group, with Lie algebra $\g$.
Write
\[
\g=\s\oplus \z,
\]
where $\s:=[\g,\g]$ and $\z$ denotes the
center of $\g$. Respectively write $S$ and $Z$ for the analytic
subgroups of $G$ corresponding to $\s$ and $\z$. By Proposition
\ref{autss}, $S$ is a Nash subgroup of $G$. Since $Z$ equals the
identity connected component of the center of $G$, it is also a Nash
subgroup of $G$.

\begin{lemt}\label{znt}
The Nash group $Z$ is a Nash torus.
\end{lemt}
\begin{proof}
Note that $Z$ is a normal subgroup of $G$. The first assertion of Lemma \ref{g12}
implies that $Z$ is reductive. Likewise $Z_\mathrm u$ is reductive
(Proposition \ref{da} implies that $Z_\mathrm u$ is a unipotent Nash
group). Then Proposition \ref{t} implies that $Z_\mathrm u=\{1\}$,
and hence $Z$ is a Nash torus by Proposition \ref{da}.
\end{proof}

Since $S\times Z$ is a finite fold cover of $G$, we prove the ``only
if" part of Theorem \ref{strr}.

On the other hand, let $G'$ be a connected Nash group with a
surjective Nash group homomorphism $H\times T\rightarrow G'$ with
finite kernel, where $H$ is a connected semisimple Nash group, and
$T$ is a Nash torus. Then $G'$ is almost linear by Proposition
\ref{quotient}.

\begin{lemt}\label{unir}
Every Nash representation of a  Nash torus is completely
reducible.
\end{lemt}
\begin{proof}
By Weyl's unitary trick, every Nash representation of an elliptic Nash group is completely reducible. In particular, every Nash representation of a compact Nash torus is completely reducible. Together with Proposition \ref{jhyp2}, this implies the lemma.

\end{proof}

By Lemma \ref{unir} and Lemma \ref{ssg}, every Nash
representation of $H\times T$ is completely reducible. Consequently, every Nash
representation of $G'$ is also completely reducible. Therefore $G'$ is reductive. This proves the ``if
part" of Theorem \ref{strr}.

By Lemma \ref{g12}, the preceding arguments also show the following

\begin{thm}\label{thmredcr}
Every Nash representation of every reductive almost linear Nash
group is completely reducible.
\end{thm}

\section{Trace forms and reductivity}

Let $G$ be an almost linear Nash group with Lie algebra $\g$. Fix a
Nash representation $V$ of $G$ with finite kernel, and write $\phi:
\g\rightarrow \gl(V)$ for the attached differential.
Put
\[
  \la x,y\ra_{\phi}:=\tr(\phi(x)\phi(y)), \qquad x,y\in \g.
\]
This defines a $G$-invariant symmetric bilinear form on $\g$, which is called the trace form attached to the Nash representation $V$.

The main result of this section is the following

\begin{thm}\label{rtr}
The almost linear Nash group $G$ is reductive if and only if the
bilinear form $\la\,,\,\ra_\phi$ is non-degenerate.
\end{thm}

Theorem \ref{rtr} has the following interesting consequence.
\begin{prpt}\label{rtr2}
Assume that $G$ is reductive. Then for every reductive Nash subgroup
$H_1$ of $G$, its centralizer $H_2$ in $G$ is also a reductive Nash
subgroup of $G$.
\end{prpt}
\begin{proof}
By Theorem \ref{rtr}, $\la\,,\,\ra_\phi$ is a non-degenerate
symmetric bilinear form on $\g$. It is $G$-invariant, and hence
$H_1$-invariant. Since $H_1$ is reductive, by Theorem \ref{thmredcr}, $\g$ is completely
reducible as a representation of $H_1$. Taking the  isotypic
decomposition, we know that the space $\g^{H_1}$ of $H_1$-fixed
vectors in $\g$ is non-degenerate with respect to
$\la\,,\,\ra_\phi$. Since $\g^{H_1}$ equals the Lie algebra of
$H_2$, the proposition follows by Theorem \ref{rtr}.

\end{proof}

The rest of this section is  devoted to a proof of Theorem
\ref{rtr}.

\begin{lemt}\label{compnd}
If $G$ is elliptic, then the bilinear form $\la\,,\,\ra_\phi$ is
negative definite. If $G$ is hyperbolic, then the bilinear form
$\la\,,\,\ra_\phi$ is positive definite.
\end{lemt}
\begin{proof}
This is implied by Lemma \ref{jgl2}.
\end{proof}

\begin{lemt}\label{orth}
Let $x$ and $y$ be two commuting elements in the Lie algebra $\gl(V)$ of $\GL(V)$. If $x$ is
elliptic and $y$ is hyperbolic, then $\tr(xy)=0$.
\end{lemt}
\begin{proof}
Note that all eigenvalues of $x$ are purely imaginary, and all
eigenvalues of $y$ are real. Since $x$ and $y$ commute, all
eigenvalues of $xy$ are purely imaginary.  Therefore $\tr(xy)$ is
purely imaginary. It has to vanish since it is also real.
\end{proof}

\begin{lemt}\label{tornd}
If $G$ is a Nash torus, then the bilinear form $\la\,,\,\ra_\phi$ is
non-degenerate.
\end{lemt}
\begin{proof}
Write  $G=T\times A$, where $T$ is a compact Nash torus, and $A$ is
a hyperbolic Nash group. Lemma \ref{orth} implies that $\Lie \,T$ and
$\Lie \,A$ are orthogonal to each other under the symmetric bilinear form
$\la\,,\,\ra_\phi$. The lemma then follows by Lemma \ref{compnd}.
 \end{proof}

\begin{lemt}\label{snd}
If $G$ is a connected semisimple Nash group and $\la\,,\,\ra_\phi$
is zero, then $G$ is trivial.
\end{lemt}
\begin{proof}
Let $K$ be a maximal compact subgroup of $G$, which is connected since $G$ is connected. Then Lemma
\ref{compnd} implies that $K$ is trivial, which further implies that
$G$ is trivial. (Recall that every non-trivial connected semisimple Lie group with finite center has a non-trivial maximal compact subgroup.)
 \end{proof}

 We are now prepared to prove the ``only if" part of Theorem \ref{rtr}:

\begin{prpt}
If $G$ is a reductive, then the bilinear form $\la\,,\,\ra_\phi$ is
non-degenerate.
\end{prpt}
\begin{proof}
Denote by $\n$ the kernel of the form $\la\,,\,\ra_\phi$. It is an
ideal of the reductive Lie algebra $\g$. Lemma \ref{tornd} implies
that $\n\subset [\g,\g]$. Therefore $\n$ is semisimple. Denote by
$N$ the analytic subgroup of $G$ with Lie algebra $\n$. It is a
connected semisimple Nash subgroup of $G$ by  Proposition
\ref{autss}. Then Lemma \ref{snd} implies that $N$ is trivial, and
hence $\n=\{0\}$.
 \end{proof}

 To prove the ``if" part of Theorem \ref{rtr}, recall the following

\begin{lemt}\label{bm}(\cite[Lemma 3.1 and Proposition 3.2]{BM})
Let $V_0$ be a finite dimensional vector space over a field of
characteristic zero. Let $\g_0$ be a Lie subalgebra of $\gl(V_0)$
such that the trace form is non-degenerate on $\g_0$. Then $\g_0$ is
reductive, and no non-zero element in the center of $\g_0$ is
nilpotent as a linear operator on $V_0$.
\end{lemt}

Now assume that $\la\,,\,\ra_\phi$ is non-degenerate. We want to
show that $G$ is reductive. In view of Lemma \ref{redi}, we may (and
do) assume that $G$ is connected.

Lemma \ref{bm} implies that the Lie algebra $\g$ is reductive. Write
\[
\g=\z\oplus \s,
\]
where $\z$ denotes the center of $\g$, and
$\s:=[\g,\g]$. Denote by $Z$ and $S$ the analytic subgroups of $G$
respectively corresponding to $\z$ and $\s$. As before, both $Z$ and $S$ are Nash subgroups of $G$.
Using Proposition \ref{da}, write  $Z=Z_\mathrm e\times Z_\mathrm
h\times Z_\mathrm u$. Then Lemma \ref{bm} implies that $Z_\mathrm
u=\{1\}$. Therefore $Z$ is a Nash torus. Since $Z\times S$ is a
finite fold cover of $G$, $G$ is reductive by Theorem \ref{strr}.
This  proves the ``if" part of Theorem \ref{rtr}.

\section{Semisimple elements}
Let $G$ be an almost linear Nash group with Lie algebra $\g$.

\begin{dfnl}
An element of  $G$ or  $\g$ is said to be semisimple if its
unipotent part is trivial.
\end{dfnl}

We define a Nash quasi-torus to be an abelian almost linear Nash group without non-trivial unipotent element. All Nash quasi-tori are reductive Nash groups.
First, we have the following
\begin{lem}\label{sst}
An element $x\in G$ is semisimple if and only if $\la x\ra$ is a
Nash quasi-torus. An element $y\in \g$ is semisimple if and only if
$\la y\ra$ is a Nash torus.
\end{lem}
\begin{proof}
The ``if" part of the first assertion is obvious. To prove the
``only if" part of the first assertion, assume that $x$ is semisimple.
Then $\la x\ra=\la x_\mathrm e\ra \times \la x_\mathrm h\ra$ by
Lemma \ref{j0}. Therefore  $\la x\ra$ is a Nash quasi-torus.
The proof of the second assertion is similar.
\end{proof}

Write $G_\mathrm{ss}$ and $\g_\mathrm{ss}$ for the sets of all
semisimple elements in $G$ and $\g$, respectively.

\begin{lem}\label{ssurj}
Let $\varphi: G\rightarrow G'$ be a Nash homomorphism of almost
linear Nash groups. Then
\[
  \varphi(G_\mathrm{ss})\subset G'_\mathrm {ss},
\]
and the inclusion becomes an equality if $\varphi$ is surjective.
Write $\phi :\g\rightarrow \g'$ for the differential of $\varphi$, where $\g'$ denotes the Lie algebra of $G'$.
Then
\[
  \phi(\g_\mathrm{ss})\subset \g'_\mathrm{ss},
\]
and the inclusion becomes an equality if $\phi$ is surjective.
\end{lem}
\begin{proof}
The proof is similar to that of Proposition \ref{pr}.
\end{proof}

The rest of this section is to prove the following
\begin{thml}\label{denses}
If $G$ is reductive, then the set $G_\mathrm {ss}$ is dense in $G$,
and the set $\g_\mathrm {ss}$ is dense in $\g$.
\end{thml}

We begin with the following
\begin{lem}\label{sl2}
If $G$ is connected and $\g$ is isomorphic to $\sl_2(\R)$, then
$G_\mathrm{ss}$ is dense in $G$.
\end{lem}
\begin{proof}
It is elementary to check that the lemma holds when $G=\SL_2(\R)$,
which implies that the lemma also holds when $G=\SL_2(\R)/\{\pm
1\}$. In general, there is a surjective Nash homomorphism $\varphi : G\rightarrow \SL_2(\R)/\{\pm 1\}$ with finite kernel. The lemma
then follows since
\[
G_\mathrm{ss}=\varphi^{-1}((\SL_2(\R)/\{\pm
1\})_\mathrm{ss}).
\]
\end{proof}

\begin{lem}\label{u}
Let $u$ be a unipotent element of a reductive Nash group $G$. Then
every neighborhood of $u$ in $G$ contains a semisimple element.
\end{lem}
\begin{proof}
The lemma is trivial when $u=1$. So assume that $u\neq 1$. Since
every element of the center of $\g$ is semisimple, $\log (u)$ belongs
to the semisimple Lie algebra $[\g,\g]$. Since $\log (u)$ is
unipotent, the linear operator
\[
\ad_{\log (u)}: [\g,\g]\rightarrow
[\g,\g],\quad x\mapsto [\log (u), x]
\]
is nilpotent. Therefore, by Jacobson-Morozov Theorem, there
is a Lie subalgebra $\g_0$ of $\g$ containing $\log (u)$ which is
isomorphic to $\sl_2(\R)$. Denote by $G_0$ the analytic subgroup of
$G$ with Lie algebra $\g_0$. It is a Nash subgroup by Proposition
\ref{autss}. The lemma then follows by Lemma \ref{sl2}.

\end{proof}

We are now ready to prove Theorem \ref{denses}. Let $x$ be an
element of a reductive Nash group $G$. The centralizer
$\oZ_G(x_\mathrm e x_\mathrm h)$ of $x_\mathrm e x_\mathrm h$ in $G$
equals the centralizer of the Nash quasi-torus $\la x_\mathrm e
x_\mathrm h\ra$ in $G$. Therefore it is a reductive Nash subgroup of
$G$ by Proposition \ref{rtr2}. Note that the product of two
commuting semisimple elements in an almost linear Nash group is
again semisimple. By Lemma \ref{u}, every neighborhood of $x_\mathrm
u$ in $\oZ_G(x_\mathrm e x_\mathrm h)$ contains a semisimple
element. Therefore every neighborhood of $x=(x_\mathrm e x_\mathrm
h) x_\mathrm u$ in $G$ contains a semisimple element. This finishes
the proof of Theorem \ref{denses} in the group case. The Lie algebra
case is proved similarly.

\section{Levi decompositions}

Let $G$ be an almost linear Nash group. Put
\[
  \mathfrak U_G:=\textrm{the identity connected component of } \bigcap_{\pi} \ker \pi,
\]
where $\pi$ runs through all irreducible Nash representations of
$G$. By Proposition \ref{subgi}, $\mathfrak U_G$ is a Nash subgroup
of $G$.

\begin{prp}\label{0unir}
The group $\mathfrak U_G$ is the largest normal unipotent  Nash
subgroup of $G$.
\end{prp}
\begin{proof}
It is obvious that $\mathfrak U_G$ is a normal subgroup of $G$. Take
a Nash representation $V$ of $G$ with finite kernel. Note that
$\mathfrak U_G$ acts trivially on all irreducible Nash representations of
$G$. Therefore, by taking a Jordan-H\"{o}lder series of $V$, we know
that $\mathfrak U_G$ acts on $V$ as unipotent linear operators. Then Lemma \ref{nilc} implies that $\mathfrak U_G$ is a unipotent Nash group.

Let $U$ be a normal unipotent  Nash subgroup of $G$. It is connected
by Proposition \ref{nashucs}. For every irreducible Nash representation
$\pi$ of $G$, the restriction $\pi|_{U}$ is completely reducible by the first assertion of
Lemma \ref{g12}. Since $U$ is unipotent, Proposition \ref{t}
implies that $U$ acts trivially on $\pi$. This shows that $U\subset
\mathfrak U_G$.

\end{proof}

We call $\mathfrak U_G$ the unipotent radical of $G$.

\begin{lemp}\label{reduni}
An almost linear Nash group is reductive if and only if its
unipotent radical is trivial.
\end{lemp}
\begin{proof}
The ``only if" part of the Lemma is obvious. The ``if" part is
implied by Proposition \ref{subgi}.
\end{proof}

Proposition \ref{subgi} also implies that $G/\mathfrak U_G$ is a
reductive Nash group.

\begin{thmp}\label{uconj}
Every reductive Nash subgroup of $G$ is contained in a maximal one,
and all maximal reductive Nash subgroups of $G$ are conjugate to
each other under $\mathfrak U_G$. Moreover, for each maximal
reductive Nash subgroup $L$ of $G$, one has that $G=L\ltimes \mathfrak U_G$.

\end{thmp}
The equality $G=L\ltimes \mathfrak U_G$ of Theorem \ref{uconj} is
called a Levi decomposition of $G$, and a maximal reductive Nash
subgroup of $G$ is called a Levi component of $G$.

The rest of this section is devoted to a proof of Theorem
\ref{uconj}. We fist recall some results of G. D. Mostow on linear
Lie algebras.

For a finite dimensional Lie algebra $\g$ over a field of
characteristic zero, write $\operatorname{Rad}(\g)$ for its
radical, namely, the largest solvable ideal of $\g$. For a Lie subalgebra $\h$  of $\g$, we define $\operatorname I_\g(\h)$ to be the
subgroup of the automorphism group $\Aut(\g)$ generated by
the set
\[
  \{\exp (\ad_x)\mid \textrm{$x\in \h$, the linear operator $\ad_x: \g\rightarrow \g, \,y\mapsto [x,y]$ is nilpotent}\}.
\]

Given a finite dimensional vector space $V$, we say that a subset $R\subset \g\l(V)$ is fully
reducible if each $R$-stable subspace of $V$ has a complementary $R$-stable subspace. This generalizes the notion of ``semisimple linear operators".

\begin{lemp}\label{mo} (\cite[Theorems 4.1 and 5.1]{Mo})
Let $V$ be a finite dimensional vector space over a field of
characteristic zero, and let $\g$ be a Lie subalgebra of $\gl(V)$.
\begin{itemize}
  \item  All maximal fully reducible Lie subalgebras of $\g$ are conjugate to each other under $\operatorname I_\g(\operatorname{Rad}([\g,\g])$.
  \item Let $R$ be a fully reducible subgroup of $\GL(V)$. If $R$ normalizes $\g$, then $R$ normalizes a maximal fully reducible Lie subalgebra of $\g$.
\end{itemize}

\end{lemp}

Now let $G$ be an almost linear Nash group as before. In the rest of this section, denote $\g$ and
$\u$  the Lie algebras of $G$ and $\mathfrak U_G$, respectively.
\begin{lemp}\label{ggu}
One has that $\operatorname{Rad}([\g,\g])\subset \u$.
\end{lemp}
\begin{proof}
By Lemma \ref{jacob2}, the Lie algebra $\g/\u$ is reductive.
Therefore
\begin{equation}\label{gg}
[\g,\g]/([\g,\g]\cap \u)\cong [\g/\u,\g/\u]
\end{equation}
is semisimple. Note that
\begin{equation}\label{gg2}
(\operatorname{Rad}([\g,\g])+ ([\g,\g]\cap \u))/([\g,\g]\cap \u)
\end{equation}
is a solvable ideal of the semisimple Lie algebra \eqref{gg}.
Therefore  \eqref{gg2} is the zero ideal, and the lemma follows.
\end{proof}

Fix a Nash homomorphism $\varphi: G\rightarrow \GL(V)$  with
finite kernel, where $V$ is a finite dimensional real vector space. Then $\g$ is identified with a Lie subalgebra of
$\gl(V)$.

\begin{lemp}\label{frc}
All maximal fully reducible Lie subalgebras of $\g$ are conjugate to
each other under $\mathfrak U_G$.
\end{lemp}
\begin{proof}
This is a direct consequence of Lemma \ref{ggu} and the first
assertion of Lemma \ref{mo}.
\end{proof}

Fix a pair $(\l,K)$ where $\l$ is a maximal fully reducible Lie
subalgebra  of $\g$, and $K$ is a maximal compact subgroup of the
normalizer $\tilde L$ of $\l$ in $G$. Lemma \ref{frc} and Theorem
\ref{compact} imply that all such pairs are conjugate to each
other under $G$. Denote by $L_0$ the analytic subgroup of $G$ with
Lie algebra $\l$.

\begin{lemp}
The subgroup $L_0$ of $G$ is a reductive Nash subgroup of $G$.
\end{lemp}
\begin{proof}
Write $L_0'$ for the smallest Nash subgroup of $G$ containing $L_0$.
It is connected since $L_0$ is so. Note that the set of $L_0$-stable
subspaces of $V$ is the same as the set of $L_0'$-stable subspaces.
Therefore $V$ is completely reducible as a representation of $L_0'$.
This implies that $L_0'$ is reductive and its Lie algebra is fully
reducible. The maximality of $\l$ then implies that $L_0=L_0'$, and the
lemma follows.
\end{proof}

Put $L:=KL_0$, which is a Nash subgroup of $G$.
We want to show that
\begin{equation}\label{glu}
  G=L\ltimes \mathfrak U_G.
\end{equation}

\begin{lemp}\label{lcirc}
One has that $L^\circ=L_0$.
\end{lemp}
\bp
Since $L_0$ is reductive, the unipotent radical $\mathfrak U_L$ of $L$ has trivial intersection with $L_0$. Then the quotient homomorphism
\[
L\rightarrow L/L_0\cong K/(K\cap L_0)
\]
restricts to an injective Nash homomorphism from $\mathfrak U_L$ to an elliptic Nash group. Therefore $\mathfrak U_L$ is trivial and $L$ is reductive.
Then the Lie algebra of $L$ is fully reducible and contains $\l$, and hence equals $\l$ by the maximality of $\l$. This proves the lemma.
\ep

Since we have proved that $L$ is reductive, we know that
\be \label{lcapu}
  L\cap \mathfrak U_G=\{1\}.
\ee

\begin{lemp}\label{sumlug}
One has that $\l+\u=\g$.
\end{lemp}
\begin{proof}
Let $s$ be a semisimple element of $\g$. By Lemma \ref{sst}, the
replica $\la s\ra$ is a Nash torus. Therefore its Lie algebra $\Lie\,
\la s\ra$ is fully reducible. By Lemma \ref{frc}, there is an
element $u\in \mathfrak U_G$ such that
\[
  s\in \Lie\, \la s\ra\subset \Ad_u(\l)\subset \Ad_u(\l+\u)=\l+\u.
\]
This proves that $\l+\u\supset \g_\mathrm{ss}$. Then Lemma \ref{ssurj}
implies that $(\l+\u)/\u\supset (\g/\u)_\mathrm{ss}$. Since
$(\g/\u)_\mathrm{ss}$ is dense in $\g/\u$ by Theorem \ref{denses}, one
knows that $(\l+\u)/\u\supset \g/\u$.  Therefore $\l+\u=\g$.
\end{proof}

Combining \eqref{lcapu} and Lemma \ref{sumlug}, we get that
\be \label{lcapu2}
  G^\circ=L_0\ltimes \mathfrak U_G\quad \textrm{and}\quad \g=\l\ltimes \u.
\ee

Recall that $\tilde L$ denotes the normalizer of $\l$ in $G$. Write
$\tilde  \l$ for its Lie algebra, and put $\u_0:=\tilde \l\cap \u$.
Then
\[
  \tilde \l=\l\times \u_0
\]
is a direct product of Lie algebras. Consequently, we have that
\begin{equation}\label{tildel}
(\tilde L)^\circ=L_0\times U_0,\quad \textrm{where }\, U_0:=\tilde L\cap \mathfrak U_G.
\end{equation}

\begin{lemp}\label{reconn}
Every connected reductive Nash subgroup of $\tilde L$ is contained
in $L_0$.
\end{lemp}
\begin{proof}
In view of \eqref{tildel}, the lemma holds
because every Nash homomorphism from a reductive Nash group to a
unipotent Nash group is trivial.
\end{proof}

\begin{lemp}\label{glu2}
One has that $G=L\ltimes \mathfrak U_G$.
\end{lemp}
\begin{proof}
By \eqref{lcapu} and \eqref{lcapu2}, it suffices to show
that every connected component of $G$ meets $K$. Since $K$ meets
every connected component of $\tilde L$, it suffices to show that
every connected component of $G$ meets $\tilde L$. Let $g\in G$.
Then by Lemma \ref{frc}, $\Ad_g(\l)=\Ad_u(\l)$ for some $u\in \mathfrak U_G$.
Therefore $u^{-1}g\in \tilde L$, and the lemma follows.
\end{proof}

Lemma \ref{glu2} implies that $L$ is a maximal reductive Nash
subgroup of $G$.

\begin{lemp}\label{containl}
Every reductive Nash subgroup $R$ of $G$ is contained in a
conjugation of $L$.
\end{lemp}
\begin{proof}
By Lemma \ref{mo}, we assume without loss of generality that
$R\subset \tilde L$. Then Lemma \ref{reconn} implies that
$R^\circ\subset L_0$. Let $K'$ be a maximal compact subgroup of $R$.
Then Theorem \ref{compact} implies that $K'\subset g K g^{-1}$ for
some $g\in \tilde L$. Therefore
\[
 R=K' R^\circ \subset g K g^{-1}
L_0=g KL_0 g^{-1}=gLg^{-1}.
\]
\end{proof}

Lemma \ref{containl} implies that all maximal reductive Nash
subgroups of $G$ are  conjugate to $L$. (Since $G=L \mathfrak U_G$, they are actually conjugate to $L$ under $\mathfrak U_G$.) This finishes the proof of
Theorem \ref{uconj}.

\section{Cartan decompositions and Iwasawa decompositions}

We first recall some basic results concerning Cartan decompositions in the setting of connected semisimple Lie groups with finite center.
\begin{prp}\label{cs}(cf. \cite[Theorem 6.31, Theorem 6.51 and Proposition 6.40]{Kn})
Let $G$ be a connected semisimple Lie group with a finite center.
Denote by $\g$ its Lie algebra. Let $K$ be a maximal compact
subgroup of $G$. Then the followings hold.
\begin{itemize}
  \item There exists a unique continuous involution $\theta_K$ of $G$ such that $G^{\theta_K}=K$.
  \item Denote by $\p$ the $-1$-eigenspace in $\g$ of the differential of $\theta_K$, then the map \[
    K\times \p\rightarrow G,\,k,x\mapsto k\exp (x)
    \]
    is a diffeomorphism.
  \item All maximal abelian subspaces of $\p$ are conjugate to each other under the adjoint action of $K$.
  \item For every $x\in \p$, the linear operator
  \[
    \ad_x: \g\rightarrow \g, \,y\mapsto [x,y]
    \]
     is semisimple and all its eigenvalues are real.
\end{itemize}
\end{prp}

Here ``involution" means an automorphism of order 1 or 2; and $G^{\theta_K}$ denotes the fixed point set of $\theta_K$ in $G$ (similar notation will be used without further explanation).

In this section, we investigate Cartan involutions for all reductive Nash groups. Let $G$ be a reductive Nash group in the rest of this section.
\begin{dfnp}
A Cartan involution of $G$ is a Nash involution of $G$ whose fixed point set is a maximal compact subgroup of $G$.
\end{dfnp}
Here ``Nash involution" means an involution which is simultaneously a Nash map. The first result of this section we intend to prove is the following
\begin{thmp}\label{cartan}
The map
\begin{equation}\label{bcartan}
  \begin{array}{rcl}
  \{\textrm{Cartan involution of $G$}\}&\rightarrow & \{\textrm{maximal
  compact subgroup of $G$}\}\\
    \theta&\mapsto & G^\theta
  \end{array}
\end{equation}
is bijective.
\end{thmp}

We begin with the following

\begin{lemp}\label{maxcrss}
Theorem \ref{cartan} holds if $G$ is a connected semisimple Nash
group or a Nash torus.
\end{lemp}
\begin{proof}
If $G$ is a connected semisimple Nash group, then all Lie group
automorphisms of $G$ are Nash automorphism. Therefore the lemma is
implied by the first assertion of Theorem \ref{cs}. If $G$ is a Nash torus, then
$G=G_\mathrm e\times G_\mathrm h$, and $G_\mathrm e$ is the unique
maximal compact subgroup of $G$. Moreover,
\[
  G_\mathrm e\times G_\mathrm h\rightarrow G_\mathrm e\times G_\mathrm
  h,\quad (x,y)\mapsto (x, y^{-1})
\]
is the unique Cartan involution of $G$. Therefore the lemma also
holds.
\end{proof}

 Denote by $\g$
the Lie algebra of $G$, and write
\[
\g=\z\oplus \s,
\]
where $\z$
denotes the center of $\g$, and $\s:=[\g,\g]$. As before, write $Z$
and $S$ for the analytic subgroups of $G$ with Lie algebras $\z$ and
$\s$, respectively. Then $Z$ is a Nash torus, and $S$ is a connected semisimple Nash group. Let $K$ be a maximal compact subgroup of $G$, and put
\[
  K_0:=K\cap S.
\]

\begin{lemp}\label{maxcr}
One has that $K\cap Z=Z_\mathrm e$, which is the unique  maximal compact subgroup of $Z$;  $K_0$ is
a maximal compact subgroup of $S$; and $K^\circ=K_0 Z_\mathrm e$.
\end{lemp}
\begin{proof}
The equality $K\cap Z=Z_\mathrm e$ is obvious. Write $\varphi:
S\times Z\rightarrow G^\circ$ for the multiplication map. It is a
finite fold covering homomorphism. Note that $K^\circ$ is a maximal
compact subgroup of $G^\circ$. Therefore $\varphi^{-1}(K^\circ)$ is
a maximal compact subgroup of $S\times Z$, which has the form
$K'_0\times Z_\mathrm e$, where $K'_0$ is a maximal compact subgroup
of $S$. We have that
\[
 K^\circ=\varphi(\varphi^{-1}(K^\circ))=K_0' Z_\mathrm e,
\]
which implies that $K_0\supset K'_0$. Since $K'_0$ is already a
maximal compact subgroup of $S$, we have that $K_0=K_0'$. This
proves the lemma.
\end{proof}

\begin{lemp}\label{injcartan}
The map \eqref{bcartan} is injective.
\end{lemp}
\begin{proof}
Let $\theta$ and $\theta'$ be two Cartan involutions of $G$ such
that $G^\theta=G^{\theta'}$. Then $S^\theta=S^{\theta'}$ and
$Z^\theta=Z^{\theta'}$. Therefore Lemma \ref{maxcrss} and Lemma
\ref{maxcr} imply that
\[
  \theta|_S=\theta'|_S\qquad\textrm{and}\qquad \theta|_Z=\theta'|_Z.
\]
The lemma then follows as $G=G^\theta SZ$.
\end{proof}

Using Lemma \ref{maxcrss} and Lemma
\ref{maxcr}, write $\theta_S$ for the unique Cartan involution of
$S$ with fix point set $K_0$. Write $\theta_Z$ for the unique
Cartan involution of $Z$.
\begin{lemp}\label{injcartan222}
There exists a unique Cartan involution of $G^\circ$ extending both $\theta_S$ and $\theta_Z$.
\end{lemp}
\begin{proof}
Uniqueness holds as $G^\circ=SZ$. Note that $S\cap Z$ is contained
in both $K_0$ and $Z_\mathrm e$. Hence $\theta_S$ and $\theta_Z$
have a common extension to an involution of $G^\circ$. One checks
that this involution has $K^\circ$ as its fixed point set, and hence
it is a Cartan involution.
\end{proof}

Denote by $\theta^\circ$ the Cartan involution of $G^\circ$ of Lemma
\ref{injcartan222}. Define a map
\begin{equation}\label{theta}
   G\rightarrow G,
  \quad kg\mapsto k\,\theta^\circ(g), \quad(k\in K, \,g\in G^\circ)
\end{equation}
It is routine to check that \eqref{theta} is a well defined Nash
involution of $G$ whose fixed point set equals  $K$. This finishes
the proof of Theorem \ref{cartan}.

Now let $\theta$ be the Cartan involution of $G$ so that $G^\theta=K$.
Still write $\theta: \g\rightarrow \g$ for its differential. Denote
by $\p$ the $-1$-eigenspace of $\theta$ in $\g$.

\begin{prp}\label{kpn}
The map
\begin{equation}\label{defsz0}
\begin{array}{rcl}
  K\times \p&\rightarrow & G,\\
   (k, x) &\mapsto & k\exp (x)
 \end{array}
 \end{equation}
 is a diffeomorphism.
\end{prp}
\begin{proof}
Without loss of generality, assume that $G$ is connected.
The second assertion of Proposition \ref{cs} as well as its analog for Nash tori imply that the map
\begin{equation}\label{defsz}
\begin{array}{rcl}
  K_0\times Z_\mathrm e\times (\p\cap \s)\times (\p\cap \z)&\rightarrow & S\times Z,\\
   (k, t, x,y) &\mapsto & (k\exp (x), t\exp (y))
 \end{array}
 \end{equation}
is a diffeomorphism. This descends to a deffeomorphism
 \begin{equation}\label{defsz}
 \nabla\backslash (K_0\times Z_\mathrm e)\times (\p\cap \s)\times (\p\cap \z)\rightarrow \nabla\backslash (S\times Z),
   \end{equation}
where
\[
  \nabla:=\{(t,t^{-1})\mid t\in  K_0\cap Z_\mathrm e=S\cap Z\}.
   \]
   The lemma then follows since the smooth map \eqref{defsz0} is obviously identified with \eqref{defsz}.
\end{proof}
\begin{lemp}\label{ph}
One has that $\p\subset \g_\mathrm h$.
\end{lemp}
\begin{proof}
Without loss of generality, assume that $G$ is connected and semisimple. Let $x\in \p$. By uniqueness of Jordan decompositions, the qualities
\[
 (-x_\mathrm e)+(-x_\mathrm h)+(-x_\mathrm u)=-x=\theta(x)=\theta(x_\mathrm e)+\theta(x_\mathrm h)+\theta(x_\mathrm e)
\]
implies that $\theta(x_\mathrm e)=-x_\mathrm e$, that is, $x_\mathrm e\in \p$. Likewise, $x_\mathrm h\in \p$ and $x_\mathrm u\in \p$.
Therefore, it suffices to show that $\p\cap \g_\mathrm e=\{0\}$ and $\p\cap \g_\mathrm u=\{0\}$.
Note that for every  $y\in \g_\mathrm e$, the linear operator $\ad_y:\g\rightarrow \g$ is semisimple and all its eigenvalues are purely imaginary.
 Together with the last assertion of Proposition \ref{cs}, this implies that $\p\cap \g_\mathrm e=\{0\}$. The equality $\p\cap \g_\mathrm u=\{0\}$ is proved similarly.
\end{proof}

\begin{prp}\label{ccartan0}
Each $\theta$-stable Nash subgroup $G_1$ of $G$ is reductive and
equals $K_1\exp (\p_1)$, where
\[
\textrm{$K_1:=G_1\cap K\quad$ and $\quad \p_1:=(\Lie\, G_1) \cap \p$.}
\]
\end{prp}
\begin{proof}
Let $g=k\exp (x)\in G_1$, where $k\in K$ and $x\in \p$. Then
\[
  \exp (2x)=(\exp (x))^2=(\exp (x) k^{-1}) (k \exp (x))=\theta(g^{-1}) g\in G_1.
\]
Then Lemma \ref{ph} and Proposition \ref{expp} imply that $\exp (\R x)\subset G_1$. Consequently,
\[
x\in \Lie\, G_1,\quad \exp (x)\in G_1,\quad \textrm{and}\quad k\in G_1.
\]
Therefore
\[
  G_1=K_1\exp (\p_1).
\]

Denote by $U_1$ the unipotent radical of $G_1$. Then it is also a
$\theta$-stable Nash subgroup of $G$. Therefore
\[
  U_1=K'_1\exp (\p'_1),\quad \textrm{where } \,K'_1:=U_1\cap K, \quad \p'_1:=(\Lie\, U_1) \cap \p.
\]
It is clear that $K'_1=\{1\}$ and $\p'_1=\{0\}$. Therefore $U_1$ is trivial and $G_1$ is reductive.
\end{proof}

The following result is an obvious consequence of the third assertion of Proposition \ref{cs}.

\begin{prp}\label{ccartan}
All maximal abelian subspaces of $\p$ are conjugate to each other under $K$.
\end{prp}

Let $\a$ be a maximal abelian subspaces of $\p$. Denote by $A$ the analytic subgroup of $G$ with Lie algebra $\a$.
\begin{prp}\label{ccartan2}
The analytic subgroup $A$ is a hyperbolic Nash subgroup of $G$.
\end{prp}
\begin{proof}
Denote by $G_1$ the centralizer of $\a$ in $G$, which is a
$\theta$-stable Nash subgroup of $G$. Note that $(\Lie\, G_1)\cap \p=\a$
since $\a$ is maximal abelian in $\p$. Therefore by Proposition
\ref{ccartan0},
\[
  G_1=K_1\exp (\a)=K_1\times A,\quad \textrm{where}\quad K_1:=G_1\cap K.
\]
Write $Z_1$ for the center of $G_1$, then $A$ equals the identity
connected component of the Nash subgroup
\[
  \{x\in Z_1\mid \theta(x)=x^{-1}\}.
\]
Therefore $A$ is a Nash subgroup.

Note that $A$ is abelian and all elements of $A$ are hyperbolic. Therefore $A$ is
hyperbolic by Proposition \ref{da}.
\end{proof}

\begin{lemp}\label{ccartan2}
The set $\exp (\p)$ is a close Nash submanifold of $G$.
\end{lemp}
\begin{proof}
The set $\exp (\p)$ is a closed submanifold of $G$ by Proposition \ref{kpn}. It is semialgebraic since it is equal to the image of the Nash map
\[
  K\times A\rightarrow G, \quad (k,a)\mapsto kak^{-1}.
\]
\end{proof}

Combining Proposition \ref{kpn}, Proposition \ref{ccartan} and Lemma \ref{ccartan2}, we obtain the following
\begin{prp}\label{kak}
One has that $G=KAK$, and the multiplication map
\[
K\times \exp (\p)\rightarrow G
\]
is a Nash diffeomorphism.
\end{prp}

Write
\[
   \g=\bigoplus_{\alpha\in \a^*} \g_\alpha,
\]
where $\a^*$ denotes the space of all real valued linear functionals on $\a$, and
\[
\g_{\alpha}:=\{x\in \g\mid [a, x]=\alpha(a)x,\,\textrm{for all }a\in \a\}.
\]
Then the set
\[
  \Delta(\g, \a):=\left\{\alpha\in \a^*\mid \alpha\neq 0,\, \g_\alpha\neq \{0\}\right\}
\]
is a root system in $\a^*$. Fix a positive system $\Delta(\g, \a)^+\subset \Delta(\g, \a)$, and put
\[
  \n:=\bigoplus_{\alpha\in \Delta(\g, \a)^+} \g_\alpha.
\]
Then $\n$ is a Lie subalgebra of $\g$. Denote by $N$ the analytic subgroup of $G$ with Lie algebra $\n$.

\begin{prp}\label{uni}
The analytic subgroup $N$ is a unipotent Nash subgroup of $G$.
\end{prp}
\begin{proof}
Without loss of generality, assume that $G$ is semisimple and
connected. Denote by $a_0$ the element of $\a$ such that
$\alpha(a_0)=1$ for all simple roots $\alpha$ in $\Delta(\g, \a)^+$.
For every integer $i$, denote
\[
  \g_i:=\{x\in \g\mid [a_0, x]=i x\}.
\]
Put
\[
  \tilde \n:=\left\{g\in \gl(\g)\mid g\left(\bigoplus_{j\geq i} \g_j\right)\subset \bigoplus_{j\geq i+1} \g_j\,\textrm{ for all $i\in \Z$}\right\},
\]
and
\[
  \tilde N:=\left\{g\in \GL(\g)\mid (g-1)\left(\bigoplus_{j\geq i} \g_j\right)\subset \bigoplus_{j\geq i+1} \g_j\,\textrm{ for all $i\in \Z$}\right\}.
\]
Then $\tilde N$ is a unipotent Nash subgroup of $\GL(\g)$ with Lie algebra $\tilde \n$.

 Consider the adjoint representation
\[
  \Ad: G\rightarrow \GL(\g)
\]
and its differential
\[
  \ad: \g\rightarrow \GL(\g).
\]
Note that $\ad^{-1}(\tilde \n)=\n$. Therefore the Nash subgroup $\Ad^{-1}(\tilde N)$ of $G$ has Lie algebra $\n$.
Hence $N$ equals the identity connected component of $\Ad^{-1}(\tilde N)$, which is a Nash subgroup of $G$.
Since $G$ is assumed to be semisimple, the adjoint representation of $N$ on $\g$ has a finite kernel. Then Lemma \ref{nilc} implies that $N$ is unipotent.
\end{proof}

\begin{thmp}\label{injcartan2}
The multiplication map \be \label{kang} K\times A\times N\rightarrow
G \ee
 is a Nash diffeomorphism.
\end{thmp}
\begin{proof}
Without loss of generality, assume that $G$ is connected. The map
\eqref{kang} is clearly a Nash map. We only need to show that it is
a diffeomorphism. This is known when $G$ is semisimple (cf. \cite[Theorem 6.46]{Kn}). The same argument as in Proposition
\ref{kpn} reduces the general case to the case when $G$ is
semisimple.

\end{proof}

As a corollary of Theorem \ref{injcartan2}, we have
\begin{prp}\label{ee22}
An almost linear Nash group is elliptic if it consists elliptic elements only.
\end{prp}
\begin{proof}
If an almost linear Nash group consists only elliptic elements, then
its unipotent radical is trivial, and is thus reductive. Then
Theorem \ref{injcartan2} implies that it is compact.

\end{proof}

The same proof as Proposition \ref{ee22} shows the following
\begin{prp}\label{ee222}
An almost linear Nash group is hyperbolic  if it consists hyperbolic elements only.
\end{prp}

\section{Exponential Nash groups}\label{secexp}

Recall from the Introduction that an almost linear Nash group $G$ is
said to be exponential if $G_\mathrm e=\{1\}$. The following lemma
is obvious.
\begin{lemd}\label{expcomp}
An almost linear Nash group is exponential if and only if all its elements are exponential.
\end{lemd}

Proposition \ref{pr} implies the following
\begin{lemd}\label{expquo}
All Nash quotient groups of exponential Nash groups are exponential Nash
groups.
\end{lemd}

Let $G$ be an almost linear Nash group. Let $K$ be a maximal compact
subgroup of $G$.

\begin{lemd}\label{expcomp}
The almost linear Nash group $G$ is exponential if and only if $K$ is trivial.
\end{lemd}
\begin{proof}
The lemma is clear since
\[
G_\mathrm e=\bigcup_{g\in G} gKg^{-1}.
\]
\end{proof}
\begin{lemd}\label{expred}
If $G$ is reductive and exponential, then $G$ is hyperbolic.
\end{lemd}
\begin{proof}
The lemma follows by Lemma \ref{expcomp} and Proposition \ref{kak}.
\end{proof}

\begin{lemd}\label{exhyp}
The almost linear Nash group $G$ is exponential if and only if
$G/\mathfrak U_G$ is a hyperbolic Nash group.
\end{lemd}
\begin{proof}
In view of Lemma \ref{expquo}, the ``only if" part is implied by
Lemma \ref{expred}. To prove the ``if" part, assume that
$G/\mathfrak U_G$ is a hyperbolic Nash group. Then under the quotient map
\[
  G\rightarrow G/\mathfrak U_G,
\]
the image of $G_\mathrm e$ is contained in
\[
  (G/\mathfrak U_G)_\mathrm e=\{1\}.
  \]
Therefore $G_\mathrm e\subset
\mathfrak U_G$, which implies that $G_\mathrm e=\{1\}$ as
$(\mathfrak U_G)_\mathrm e=\{1\}$.
\end{proof}

Using Levi decompositions, Lemma \ref{exhyp} implies that every
exponential Nash group is connected, simply connected, and solvable.

\begin{lemd}\label{uhcoc}
If $G$ is unipotent or hyperbolic, then there is no proper co-compact Nash subgroup of $G$.
\end{lemd}
\begin{proof}
The hyperbolic case is obvious. Assume that $G$ is unipotent. We prove the lemma by induction on $\dim G$.
It is trivial when $\dim G=0$. Assume that $\dim G>0$ and the lemma holds for unipotent Nash groups of smaller dimension.

Let $H$ be a co-compact Nash subgroup of $G$. Denote by $Z$ the
center of $G$. It is a Nash subgroup of $G$ of positive dimension.
Note that $ZH$ is a Nash subgroup of $G$, and  $ZH/H$ is a closed
subset of $G/H$. Therefore
\[
 Z/(Z\cap H)=ZH/H
 \]
 is compact. Since the
lemma obviously holds for abelian unipotent Nash groups, we have
that $Z\cap H=Z$, or equivalently, $H\supset Z$. Then $H/Z$ is a
co-compact Nash subgroup of the unipotent Nash group $G/Z$. Since
$\dim G/Z<\dim G$, by the induction hypothesis, we have that $H/Z=G/Z$, in other words, $H=G$.

\end{proof}

\begin{lemd}\label{uhcoc2}
If $G$ is exponential, then there is no proper co-compact Nash subgroup of $G$.
\end{lemd}
\begin{proof}
Let $H$ be a co-compact Nash subgroup of $G$. Using Proposition
\ref{closed}, we get a closed orbit $O\subset G/H$ under left
translations by $\mathfrak U_G$. Since $O$ is  compact, Lemma
\ref{uhcoc} implies that $O$ has only one point, say $g_0 H$. Then $\mathfrak U_G g_0 H\subset g_0 H$, which implies that
$\mathfrak U_G\subset H$ as $\mathfrak U_G$ is a normal subgroup
of $G$. Now $H/\mathfrak U_G$ is a co-compact Nash subgroup of the
hyperbolic Nash group $G/\mathfrak U_G$. Lemma \ref{uhcoc} implies
that $H/\mathfrak U_G=G/\mathfrak U_G$, in other words, $H=G$.
\end{proof}

The following is Borel fixed point theorem in the setting of Nash groups.
\begin{thmd}\label{bfix}
Let $G\times M\rightarrow M$ be a Nash action of $G$ on a non-empty
Nash manifold $M$. If $G$ is exponential and $M$ is compact, then
the action has a fixed point.
\end{thmd}
\begin{proof}
Using Proposition \ref{closed}, we get a closed $G$-orbit $O\subset
M$.  Then $O$ is compact and Lemma \ref{uhcoc2} implies that $O$ has
only one point.
\end{proof}

\begin{lemd}\label{b}
There exists an exponential Nash subgroup $B$ of $G$ such that the
multiplication map $K\times B\rightarrow G$ is a Nash
diffeomorphism.
\end{lemd}
\begin{proof}
Without loss of generality, assume that $G$ is reductive (otherwise,
take a Levi component of $G$ containing $K$). Then the group $AN$ of
Theorem \ref{injcartan2} fulfills the requirement of the lemma.
\end{proof}

\begin{lemd}\label{cocompacte}
If $G$ is not exponential, then $G$ has a proper co-compact Nash
subgroup.
\end{lemd}
\begin{proof}
The group $B$ of lemma \ref{b} is a proper co-compact Nash subgroup
of $G$.
\end{proof}

Recall the following
\begin{lemd}\label{subbn} (\cite[Section I.1, Theorem
1]{LL}) Let $H$ be a  connected, simply connected, solvable Lie
group with Lie algebra $\h$. If the the exponential map $\exp:
\h\rightarrow H$ is either injective or surjective, then it is a
diffeomorphism.
\end{lemd}

Denote by $\g$ the Lie algebra of $G$.

\begin{prpd}\label{expiso}
The almost linear Nash group $G$ is exponential if and only if the exponential map
\begin{equation}\label{exph2}
\exp: \g\rightarrow G
\end{equation}
is a diffeomorphism.
\end{prpd}
\begin{proof}
The ``only if" part is implied by Proposition \ref{expp0} and Lemma
\ref{subbn}. To prove the ``if" part of the proposition, assume that
\eqref{exph2} is a diffeomorphism. Then $G$ is connected. Therefore
$K$ is connected and the exponential map
\begin{equation}\label{exph3}
\exp: \k\rightarrow K
\end{equation}
is injective, where $\k$ denotes the Lie algebra of $K$. This forces $K$ to be trivial. Therefore $G$ is exponential by Lemma \ref{expcomp}.
\end{proof}

\begin{lemd}\label{maxdh}
For each co-compact Nash subgroup $H$ of $G$, one has that $\dim
H\geq \dim G/K$.
\end{lemd}
\begin{proof}
Let $B$ be as in Lemma \ref{b}. By Theorem \ref{bfix}, the left translation action of $B$ on $G/H$ has a fixed point, say $g_0 H$. Then $Bg_0H\subset g_0H$, which implies that
\[
  \dim H\geq \dim B =\dim G/K.
  \]
\end{proof}

Denote by $\oB_n(\R)$ the Nash subgroup of $\GL_n(\R)$ consisting
all upper-triangular matrices with positive diagonal entries ($n\geq
0$). Its Lie algebra $\b_n(\R)$ consisting all upper-triangular
matrices in $\gl_n(\R)$. It is obvious that $\oB_n(\R)$ and all its
Nash subgroups are exponential Nash groups. Conversely, we have

\begin{lemd}\label{bnr}
Every exponential Nash group $H$ is Nash isomorphic to a Nash
subgroup of $\oB_n(\R)$ for some $n\geq 0$.
\end{lemd}
\begin{proof}
Fix a Nash representation $V$ of $H$ with finite kernel. By Lemma
\ref{expcomp}, the representation is actually faithful. Consider the
induced action of $H$ on the compact Nash manifold of all full flags
in $V$. Then Theorem \ref{bfix} implies that the action has a fixed
point, that is, $H$ stabilizes a full flag in $V$. Therefore there
exists an injective Nash homomorphism $\varphi: H\rightarrow
\oB'_n(\R)$, where $n:=\dim V$, and $\oB'_n(\R)$ denotes the Nash
subgroup of $\GL_n(\R)$ of upper-triangular matrices.  Since $H$ is
connected, $\varphi(H)$ is contained in $\oB_n(\R)$ and the lemma
follows.
\end{proof}
\begin{thmd}\label{conjexp}
Every exponential Nash subgroup of $G$ is contained in a maximal one, and all maximal exponential Nash subgroups of $G$ are conjugate to each other in $G$.
\end{thmd}
\begin{proof}
Let $B$ be as in Lemma \ref{b}. Let $H$ be an exponential Nash subgroup of $G$. By Theorem \ref{bfix}, the left translation action of $H$ on $G/B$ has a fixed point, say, $g_0 B$. Then $Hg_0 B\subset g_0 B$,  and consequently, $H$ is contained in a conjugation of $B$. Therefore, $B$ has the largest dimension among all exponential Nash subgroup of $G$.
In particular, $B$ is a maximal exponential Nash subgroup of $G$ (since all exponential Nash groups are connected). This proves the theorem.

\end{proof}

\begin{thmd}\label{gkb}
A Nash subgroup $B$ of $G$ is a   maximal exponential Nash subgroup if and only if the multiplication map $K\times B\rightarrow G$ is a Nash diffeomorphism.
\end{thmd}
\begin{proof}
Let $B$ be a Nash subgroup of $G$. We first prove the ``if" part of the theorem. So assume that the multiplication map $K\times B\rightarrow G$ is a Nash diffeomorphism. Then $B$ is connected. Let $K'$ be a maximal compact subgroup of $B$. Applying Lemma
\ref{b} to $B$, we get an exponential Nash subgroup $B'$ of $B$ so
that the multiplication map $K'\times B'\rightarrow B$ is a Nash
diffeomorphism. Then $B'$ is co-compact in $G$. Hence by Lemma \ref{maxdh},
 \[
  \dim B'\geq \dim G/K=\dim B.
  \]
 Therefore  $B'=B$, and $B$ is an exponential Nash subgroup of $G$. Then the proof of Theorem \ref{conjexp} shows that $B$ is a maximal  exponential Nash subgroup of $G$.

To prove the ``only if" part of the theorem, assume that $B$ is a maximal exponential Nash subgroup of $G$. Using Lemma \ref{b}, take an exponential Nash subgroup $B_0$ of $G$
so that the multiplication map
\begin{equation}\label{kb0}
K\times B_0\rightarrow G\,\,\textrm{ is a Nash diffeomorphism.}
\end{equation}
The proof of Theorem \ref{conjexp} shows that $B_0$ is a maximal
exponential Nash subgroup of $G$. By Theorem \ref{conjexp},
\[
  B=g_0 B_0 g_0^{-1}\quad \textrm{for some $g_0\in G$}.
\]
Write $g_0=k_0 b_0$, where $k_0\in K$ and $b_0\in B_0$. Note that \eqref{kb0} implies that the multiplication map
\[
  K\times k_0 B_0 k_0^{-1}\rightarrow G
\]
is a Nash diffeomorphism. The ``only if" part of the theorem then follows as $k_0 B_0 k_0^{-1}=B$.
\end{proof}

\begin{thmd}\label{conhyp}
Every hyperbolic Nash subgroup of $G$ is contained in a maximal one,
and all maximal hyperbolic Nash subgroups of $G$ are conjugate to
each other in $G$.
\end{thmd}
\bp Fix a maximal exponential Nash subgroup $B$ of $G$, and fix a
Levi component $A$ of $B$. Let $H$ be a hyperbolic Nash subgroup of
$G$. Then by Theorem \ref{conjexp}, a conjugation of $H$ is
contained in $B$. Theorem \ref{uconj} further implies that a
conjugation of $H$ is contained in $A$. As in the proof of Theorem \ref{conjexp}, we know that $A$ is a maximal hyperbolic Nash subgroup by dimension reason.
This proves the theorem.

\ep

\begin{lemd}\label{expuni}
If $G$ is exponential, then every unipotent Nash subgroup of $G$ is
contained in $\mathfrak U_G$.
\end{lemd}
\bp Let $U$ be a unipotent Nash subgroup of $G$. Then the quotient
homomorphism
\[
  G\rightarrow G/\mathfrak U_G
\]
has trivial restriction to $U$. Therefore $U\subset \mathfrak U_G$.
 \ep

In view of Lemma \ref{expuni}, a similar argument as Theorem
\ref{conhyp} implies the following

\begin{thmd}\label{conuni}
Every unipotent Nash subgroup of $G$ is contained in a maximal one,
and all maximal unipotent Nash subgroups of $G$ are conjugate to
each other in $G$.
\end{thmd}

By the preceding arguments, we know that for each maximal
exponential Nash subgroup $B$ of $G$, its unipotent radical
$\mathfrak U_B$ is a maximal unipotent Nash subgroup of $G$, and
each Levi component of $B$ is a maximal hyperbolic Nash subgroup of
$G$.

\section{About proofs of the results in the Introduction}

In this last section, we collect some results of previous sections
to explain the proofs of the propositions and theorems which occur
in the Introduction.

Proposition \ref{quotient0} is a restatement of Proposition
\ref{quotient}.

Recall that Proposition \ref{criehu} asserts the following: An
almost linear Nash group is elliptic, hyperbolic, or unipotent if
and only if all of its elements are elliptic, hyperbolic, or
unipotent, respectively. The ``only if" part of the proposition is
trivial. The elliptic case and hyperbolic case of the ``if" part is
proved in Propositions \ref{ee22} and \ref{ee222}, respectively. To
prove the ``if" part in the unipotent case, let $G$ be an almost
linear Nash group consisting unipotent elements only. Then $G$ is
exponential. Hence a Levi component of $G$ is a hyperbolic Nash group, which
has to be trivial. Therefore $G$ is unipotent. This finishes the
proof of Proposition \ref{criehu}.

Proposition \ref{thmehu} consists two assertions. The first one is
the following
\begin{prpd}
Let $G$ be an almost linear Nash group which is elliptic, hyperbolic or unipotent. Then all of its Nash subgroups and Nash quotient groups are elliptic,
hyperbolic or unipotent, respectively.
\end{prpd}

\begin{proof}
The assertion for Nash subgroups is obvious. The assertion for Nash
quotient groups appears in Propositions \ref{quotientel},
\ref{quosplit} and \ref{quotientunip}.

\end{proof}

The second assertion of Proposition \ref{thmehu} is the following

\begin{prpd}
Let $G$ be an almost linear Nash group. If $G$ has a normal Nash subgroup $H$ so
that $H$ and $G/H$ are both elliptic, both hyperbolic or both unipotent, then
$G$ is  elliptic, hyperbolic or unipotent, respectively.
\end{prpd}
\bp

Assume that both $H$ and $G/H$ are elliptic. The the image of
$G_\mathrm h$ under the quotient map
\[
  G\rightarrow G/H
\]
is contained in $(G/H)_\mathrm h=\{1\}$. Therefore $G_\mathrm
h\subset H$, which implies that $G_{\mathrm h}=\{1\}$. Similarly
$G_{\mathrm u}=\{1\}$. Therefore $G=G_\mathrm e$, and Proposition
\ref{criehu} implies that $G$ is elliptic.

The same argument proves the proposition in the hyperbolic and
unipotent case.

 \ep

As already mentioned, Theorem \ref{thmc} is a combination of Lemmas
\ref{compactn}, \ref{auts2} and \ref{auts}, and Theorem \ref{thmu}
is a combination of Propositions \ref{nashucs}, \ref{autsu2}, \ref{unin} and \ref{nashucs000}. Theorem \ref{thmh} is a restatement of Theorem
\ref{equsplit}. Theorem  \ref{thm2} is the same as Theorem \ref{jd}.

Theorem \ref{thmcs} consists five assertions. The second one is
obvious. The others are respectively proved in Lemma \ref{ssl},
Proposition \ref{snash}, Proposition \ref{autss2} and Proposition
\ref{autss}.

For Theorem \ref{thmred}, it is obvious that (b) $\Rightarrow$ (a),
and Theorem \ref{thmredcr} asserts that (a) $\Rightarrow$ (b). Lemma
\ref{reduni} assert that (a) $\Leftrightarrow$ (c). Theorem \ref{rtr}
implies that
\[
   (d)\Leftrightarrow (a)\Leftrightarrow(e).
\]
The equivalence (a) $\Leftrightarrow$ (f) is proved in Lemma
\ref{redi}, and (a) $\Leftrightarrow$ (g) is proved in Theorem
\ref{strr}. Therefore Theorem \ref{thmred} holds.

For Theorem \ref{thmexp}, (a) $\Leftrightarrow$ (b) is implied by
Lemma \ref{expcomp}, and (a) $\Leftrightarrow$ (c) is implied by
Lemmas \ref{uhcoc2} and \ref{cocompacte}. The equivalence
(a) $\Leftrightarrow$ (d) is proved in Lemma \ref{exhyp},
(a) $\Leftrightarrow$ (e) is implied by Lemma \ref{bnr},  and
(a) $\Leftrightarrow$ (f) is proved in Proposition \ref{expiso}. By
Theorem \ref{bfix}, (a) $\Rightarrow$ (g), and by Lemma
\ref{cocompacte}, (g) $\Rightarrow$ (a). In conclusion,  Theorem
\ref{thmexp} holds.

Theorem \ref{conall} is contained in Theorems \ref{compact},
\ref{conhyp}, \ref{conuni}, \ref{uconj} and \ref{conjexp}.

Finally, Theorem \ref{levd0} is contained in Theorem \ref{uconj},
and Theorem \ref{iwasawa} is contained in Theorem \ref{gkb}.


\begin{thebibliography}{99}


\bibitem[BOR]{BOR}
S. Basu, R. Ollack and M.-F. Roy, \textit{Algorithms in Real Algebraic Geometry}, Algorithms and Computation in Mathematics, Springer-Verlag, Berlin, 2003.

\bibitem[BCR]{BCR}
J. Bochnak, M. Coste and M.F. Roy, \textit{Real Algebraic Geometry},
Ergebnisse der Math., Vol. 36, Springer, Berlin, 1998.

\bibitem[Bo1]{Bor91}
 A. Borel, \textit{Linear Algebraic Groups}, 2nd ed., Graduate Texts in Mathematics, vol. 126, Springer-Verlag, New York, 1991.

\bibitem[Bo2]{Borel}
A. Borel, \textit{Semisimple Groups and Riemannian Symmetric
Spaces}, Texts and Readings in Mathematics 16,  Hindustan book agency, 1998.

\bibitem[BM]{BM}
A. Borel and G. D. Mostow, \textit{On semi-simple automorphisms of
Lie algebras}, Ann. Math., 61 (1955), 389-504.

\bibitem[Bo]{Bo}
N. Bourbaki, \textit{Alg\`{e}bre, Chapitres 4-7}, Masson, Paris,
1981.

\bibitem[Ca]{Car}
P. Cartier, \textit{A primer of Hopf algebras}, Frontiers in number
theory, physics, and geometry II. On conformal field theories,
discrete groups and renormalization. Papers from the meeting, Les
Houches, France, March 9-21, 2003, 537-615. Springer, Berlin 2007.



\bibitem[Co]{Co}
M. Coste, \textit{An Introduction to Semialgebraic Geometry}, RAAG
Notes, Institut de Recherche Math¨¦matiques de Rennes, 2002.

\bibitem[Di]{Dix}
J. Dixmier, \textit{L¡¯application exponentielle dans les groupes de
Lie r¡äesolubles}, Bull. Soc. Math. France 85 (1957), 113-121.



\bibitem[GW]{GW}
R. Goodman and N.R. Wallach, \textit{Symmetry, Representations and
Invariants}, Graduate Texts in Mathematics 255, Springer, New York,
2009.


\bibitem[He]{He}
S. Helgason, \textit{Differential Geometry, Lie Groups, and
Symmetric Spaces}, American Mathematical Society, 2001.

\bibitem[Ja]{Jac}
N. Jacobson, \textit{Completely reducible Lie algebras of linear
transformations}, Proc. Amer. Math. Soc. 2 (1951), 105-113.

\bibitem[Kn]{Kn}
A. W. Knapp, \textit{Lie Groups Beyond an Introduction},
Birkh\"{a}user, Boston, 2nd ed. 1996.


\bibitem[LL]{LL}
H. Leptin and J. Ludwig, \textit{Unitary Representation Theory of
Exponential Lie Groups}, Berlin, W. de Gruyter, 1994.



\bibitem[Mi]{Mil}
J.S. Milne, \textit{Lie Algebras, Algebraic Groups, and Lie Groups},
2012, book available at
http://www.jmilne.org/math/CourseNotes/LAG.pdf

\bibitem[Mo]{Mo}
G. D. Mostow, \textit{Fully reducible subgroups of algebraic
groups}, Amer. J. Math. 78 (1956), 200-221.

\bibitem[Sh1]{Sh}
M. Shiota, \textit{Nash Manifolds}, Lect. Notes Math., vol. 1269,
Springer-Verlag, 1987.

\bibitem[Sh2]{Sh2}
M. Shiota, \textit{Nash functions and manifolds} in \textit{Lectures
in Real Geometry}, F. Broglia (edit.) W.de Gruyter , Berlin, New
York, 1996.

\bibitem[Wa]{Wa}
F.W. Warner, \textit{Foundations of Differentiable Manifolds and Lie Groups}, Graduate Texts in Mathematics 94, Springer, 1983.


\end{thebibliography}
\end{document}